\documentclass[a4paper,11pt]{amsart}

\usepackage{amsmath,amssymb,amsfonts,amsthm}

\usepackage{graphicx}
\usepackage{tikz}
\usetikzlibrary{arrows}

\usepackage{enumerate}
\usepackage[shortlabels]{enumitem}

\usepackage{xcolor}
\usepackage{hyperref}

\usepackage{float}
\usepackage{ulem}
\usepackage{csquotes}

\usepackage[
left=2.22cm,
right=2.22cm,
top=2.9cm,
bottom=2.9cm
]{geometry}

\numberwithin{equation}{section}

\newtheorem{assumption}{Assumption}[section]
\newtheorem{theorem}{Theorem}[section]

\newtheorem{proposition}[theorem]{Proposition}
\newtheorem{corollary}[theorem]{Corollary}
\newtheorem{remark}[theorem]{Remark}

\theoremstyle{definition}
\newtheorem{definition}{Definition}[section]

\hypersetup{
	colorlinks=true,
	linkcolor=blue,
	citecolor=blue,
	urlcolor=blue
}

\title[Quintic Wave with Nonlocal Kelvin--Voigt Damping]{Attractors and Singular Limits for a Quintic Wave Equation with
	Nonlocal Kelvin--Voigt Damping}

\author[Yue Sun]{Yue Sun$^{\star}$}
\thanks{$^{\star}$ Research partially supported by the Natural Science Foundation of Jiangsu Province, Grant No.~BK20240532.}
\address[Yue Sun]{School of Physical and Mathematical Sciences, Nanjing Tech University, Nanjing 211816, China}
\email{S\_yyue@163.com}

\author[Marcelo M. Cavalcanti]{Marcelo M. Cavalcanti$^{\ddagger}$}
\thanks{$^{\ddagger}$ Research partially supported by CNPq Grant No.~300631/2003-0.}
\address[Marcelo M. Cavalcanti]{State University of Maring\'a, Maring\'a, PR 87020-900, Brazil}
\email{mmcavalcanti@uem.br}

\author[Vando Narciso]{Vando Narciso$^{\dagger}$}
\thanks{$^{\dagger}$ Research partially supported by Fundect/CNPq Grant No.~15/2024.}
\thanks{Corresponding author.}
\address[Vando Narciso]{Universidade Estadual de Mato Grosso do Sul, Dourados 79804-970, MS, Brazil}
\email{vnarciso@uems.br}

\date{\today}

\begin{document}
	
	\begin{abstract}
		In this article, we consider an energy-critical quintic wave equation on a bounded domain $\Omega\subset\mathbb{R}^3$ with nonlinear and nonlocal Kelvin--Voigt damping of the form $-\|\nabla u_t\|_{L^2(\Omega)}^\alpha\Delta u_t$, where $\alpha\in\mathbb{R}_+=[0,\infty)$. Under suitable hypotheses on the quintic source term, we establish the well-posedness of the problem and investigate its long-time dynamics in the natural energy space $\mathcal H=H_0^1(\Omega)\times L^2(\Omega)$. For every $\alpha\in\mathbb{R}_+$, we show that the associated dynamical system $(\mathcal H,S^\alpha(t))$ is gradient and dissipative, and we prove a stabilization estimate that yields asymptotic smoothness and, consequently, the existence of a compact global attractor $\mathcal A_\alpha$. The same estimate provides an upper bound for the Kolmogorov $\varepsilon$-entropy of $\mathcal A_\alpha$ and, in the limiting case $\alpha=0$, reduces to a quasi-stability inequality, which implies that $\mathcal A_0$ has finite fractal dimension. Furthermore, we prove that the family $\{\mathcal A_\alpha\}_{\alpha\in\mathbb{R}_+}$ is uniformly bounded in the higher-regularity space $\mathcal H_1=(H^2(\Omega)\cap H_0^1(\Omega))\times H_0^1(\Omega)$. Finally, we establish the upper semicontinuity of $\{\mathcal A_\alpha\}_{\alpha\in\mathbb{R}_+}$ at $\alpha=0$, showing that the attractors associated with the nonlinear and nonlocal Kelvin--Voigt damping converge to the global attractor of the limiting problem with classical linear Kelvin--Voigt damping.

		
		
	\end{abstract}	
	\maketitle	
	\section{Introduction}
	\subsection{Model Description and Main Results}
	
	Let $\Omega\subset\mathbb{R}^3$ be a bounded domain with smooth boundary
	$\Gamma=\partial\Omega$. In this paper, we investigate the following
	energy-critical quintic wave equation with nonlinear nonlocal
	Kelvin--Voigt damping:
	\begin{equation}\label{P}
		\left\{
		\begin{aligned}
			&u_{tt}-\Delta u-\|\nabla u_t(t)\|^{\alpha}\Delta u_t+f(u)=h,
			&&\text{in } \Omega\times\mathbb{R}_{+},\\
			&u=0,
			&&\text{on } \Gamma\times\mathbb{R}_{+},\\
			&u(\cdot,0)=u_0,\qquad
			u_t(\cdot,0)=u_1,
			&&\text{in } \Omega.
		\end{aligned}
		\right.
	\end{equation}
	Here, $\alpha\in \mathbb{R}_+$ is a prescribed parameter, $h$ is a given external force,
	and $\|\cdot\|$ denotes the norm in $L^2(\Omega)$. The nonlinearity
	$f:\mathbb{R}\to\mathbb{R}$ is assumed to satisfy the standard
	energy-critical quintic growth conditions. Problem \eqref{P} combines two challenging features, namely, an energy-critical source term and a nonlinear, nonlocal Kelvin--Voigt damping. It thus lies at the intersection of two active areas of research: the theory of energy-critical wave equations and the analysis of nonlocal dissipative mechanisms in evolutionary PDEs.

	The principal objective of this work is study the well-posedness
	and long-time dynamical theory for problem \eqref{P}. First, we establish the
	existence of both global strong and weak solutions, together with
	continuous dependence on the initial data in the natural energy space
	$
	\mathcal H=H_0^1(\Omega)\times L^2(\Omega).
	$
	The analysis relies on the Galerkin approximation method, compactness
	arguments based on the Aubin--Lions theorem, and suitable density
	techniques. We then investigate the asymptotic behavior of the associated
	solution semigroup $S^\alpha(t):\mathcal H\to\mathcal H$ generated by the
	weak solutions of problem \eqref{P}. We prove that, for every
	$\alpha\in\mathbb{R}_+$, the corresponding dynamical system
	$(\mathcal H,S^\alpha(t))$ possesses a bounded absorbing set in $\mathcal H$
	and is gradient and asymptotically smooth. Consequently, it admits a
	compact global attractor
	$\mathcal A_\alpha=\mathcal M^u_{\alpha}(\mathcal N)$,
	where $\mathcal N$ denotes the set of stationary equilibria and
	$\mathcal M^u_{\alpha}(\mathcal N)$ is the corresponding unstable manifold.
	Moreover, we prove that the family of global attractors
	$\{\mathcal A_\alpha\}_{\alpha\in\mathbb{R}_+}$ is uniformly bounded in the
	more regular phase space
	$\mathcal H_1=(H^2(\Omega)\cap H_0^1(\Omega))\times H_0^1(\Omega)$.
	Finally, we study the dependence of the global attractors on the parameter
	$\alpha$ and prove that the family
	$\{\mathcal A_\alpha\}_{\alpha\in\mathbb{R}_+}$ is upper semicontinuous at
	$\alpha=0$, namely,
	$$
	d_{\mathcal H}(\mathcal A_\alpha,\mathcal A_0)
	=
	\sup_{\psi\in\mathcal A_\alpha}
	\inf_{\sigma\in\mathcal A_0}
	\|\psi-\sigma\|_{\mathcal H}
	\longrightarrow 0,
	\qquad\text{as }\alpha\to0.
	$$
	It is worth noting that the case $\alpha=0$ corresponds to the classical
	Kelvin--Voigt damping, for which problem \eqref{P} reduces to
	\begin{equation}\label{strongly-damped}
		u_{tt} - \Delta u - \Delta u_t + f(u) = h,
	\end{equation}
	Hence, the model considered here provides a nonlinear nonlocal extension of
	the classical Kelvin--Voigt equation while preserving its limiting
	dynamical behavior as $\alpha\to0$.
	
	\subsection{A Brief Review on Energy-Critical Wave Equations}
	The theory of quintic wave equations has been extensively developed over the past decades. The pioneering global existence result in the whole space $\mathbb{R}^3$ was obtained by Struwe~\cite{Struwe} for radially symmetric initial data. This was later generalized by Grillakis to the nonradial case for smooth initial data. Both approaches make essential use of explicit solution representations for the three-dimensional wave equation, combined with the Morawetz--Pohozaev identity. A complete well-posedness theory in the natural energy space was subsequently established by Shatah and Struwe~\cite{Shatah-Struwe-93}, with additional developments presented in~\cite{Struwe-94}. The corresponding class of Shatah--Struwe solutions enjoys improved space--time regularity, which has become a fundamental tool in the study of the long-time dynamics of energy-critical wave equations.
	
	The investigation of asymptotic behavior for dissipative wave equations with quintic critical nonlinearities in bounded three-dimensional domains began with the pioneering contributions of
	\cite{Carvalho-Cholewa,Pata-Zelik,Yang-Sun}. These papers considered strongly damped wave equations with classical
	Kelvin--Voigt damping, of the form \eqref{strongly-damped} and established fundamental results on their long-time dynamics.
	
	In the framework of bounded smooth domains \(\Omega\subset\mathbb{R}^n\), the long-time dynamics of wave equations with critical nonlinearities,
	$
	f(u)\sim |u|^{\frac{n+2}{n-2}},
	$
	which corresponds to the quintic growth \(f(u)\sim |u|^5\) in dimension \(n=3\), has also been the subject of extensive investigation. In particular, Carvalho and Cholewa~\cite[Section~4.3]{Carvalho-Cholewa} proved the existence of compact global attractors by exploiting Alekseev's nonlinear variation-of-constants formula. Subsequently, Pata and Zelik~\cite{Pata-Zelik} refined these results by establishing the existence of compact global attractors together with their optimal regularity properties.
	
	Further developments on global well-posedness, dissipativity, and the existence
	of attractors for quintic wave equations with fractional structural damping were
	obtained in
	\cite{KSZ,SavostianovZelik2014a,SavostianovZelik2014b,Savostianov-thesis}.
	More precisely, the authors considered equations of the form
	\begin{align*}
		u_{tt}-\Delta u+\gamma(-\Delta)^{\theta}u_t+f(u)=h,
	\end{align*}
	posed in bounded domains \(\Omega\subset\mathbb{R}^3\), with
	$
	\theta=0$ or $\theta=\frac12.
	$ In the case $\theta=0$, the approach developed in
	\cite[Chapter~2]{SavostianovZelik2014a} and \cite{KSZ}
	combines Strichartz estimates in bounded domains with suitable extensions of
	the Morawetz--Pohozaev identity. This framework allows the authors to establish
	the global existence of Shatah--Struwe solutions and the existence of a compact
	global attractor for the associated dynamical system.
	
	For the structural damping case \(\theta=\frac12\), the analysis developed in
	\cite[Chapter~4]{Savostianov-thesis} and \cite{SavostianovZelik2014a}
	relies on a suitable Lyapunov-type functional to obtain additional regularity
	of energy solutions. As a consequence, the authors proved global
	well-posedness, dissipativity, and the existence of smooth global and
	exponential attractors with finite Hausdorff and fractal dimensions. Moreover,
	the hidden regularizing effects induced by the structural damping were further
	analyzed in \cite{SavostianovZelik2014b}, leading to the construction of smooth
	attractors in the case \(\theta=\frac12\).
	
	Motivated by these developments and by the successful application of
	Strichartz estimates to the analysis of quintic wave equations, subsequent
	works have focused on dissipative models whose damping mechanisms are compatible with dispersive estimates. In this direction, Cavalcanti \textit{et al.}~\cite{Cavalcanti} investigated the quintic wave equation with localized damping of the form \(a(x)u_t\). Combining Strichartz estimates with the Galerkin approximation method, the authors established the global
	existence of Shatah-Struwe solutions and proved exponential energy decay by means of a contradiction argument. More recently, in the framework of Kelvin--Voigt dissipation, Cavalcanti and D. Cavalcanti~\cite{M. Cavalcanti-V. Cavalcanti} considered the energy-critical quintic wave equation with localized damping
	$$
	u_{tt}-\Delta u-\operatorname{div}(a(x)\nabla u_t)+u^5=0,
	$$
	where the damping coefficient \(a(x)\geq0\) is localized in space. Their analysis established the global well-posedness of finite-energy solutions and provided stability results for the evolution generated by the problem.
	
	For classes of quintic wave equations whose dissipative mechanisms are not compatible with Strichartz estimates, different techniques are required. In this direction, Lasiecka and Narciso~\cite{LasieckaNarciso2025} studied a quintic wave equation in a bounded three-dimensional domain with nonlinear damping of the form
	$
	g(u_t)\sim |u_t|^p u_t,$ $p=4.
	$
	The authors proved the existence of a compact global attractor and showed that, under the additional assumption \(g'(0)>0\), the associated dynamical system is quasi-stable. As a consequence, the attractor has finite fractal dimension, possesses additional regularity, and admits an exponential attractor. In this energy-critical setting, the analysis relies on refined energy-dissipation mechanisms, the derivation of suitable energy identities for weak solutions, appropriate modifications of Ball's method, and the general theory of quasi-stable dynamical systems. Related dissipative models have recently been investigated in
	\cite{LiuMengHanZhang2023,LiuMengZhang2025,Zhou2025}.
	
	The nonlinear nonlocal Kelvin--Voigt damping considered in problem~\eqref{P}
	belongs to a broader class of dissipative mechanisms usually referred to as
	\textit{averaged damping}. To place our model within this framework, we briefly
	review below some relevant developments on wave equations with averaged
	damping.
	
	\subsection{Wave Equations with Averaged Damping}
	
	Wave equations involving dissipative terms with nonlocal coefficients depending
	on the velocity of the system were investigated by Aloui \textit{et al.}
	\cite[Section~8]{ABH}. In this setting, the dissipative coefficient is obtained
	from global quantities associated with the state of the system, and the
	corresponding mechanism is commonly referred to as averaged damping. Within this class of dissipative mechanisms, Yan \textit{et al.}~\cite{YTZ}
	studied the wave equation
	\begin{align}\label{model_averaged_damping-1}
		u_{tt}-\Delta u+k\|u_t(t)\|^p u_t+f(u)=g(x),
		\quad \text{in } \Omega\times\mathbb{R}_{+},
	\end{align}
	subject to homogeneous Dirichlet boundary conditions. Here,
	$\Omega\subset\mathbb{R}^3$ is a bounded domain with smooth boundary,
	$k,p>0$ are fixed constants, and
	$g\in L^2(\Omega)$ is an external force. The nonlinear source term satisfies
	the subquintic growth condition
	$$
	|f'(s)|\leq C(1+|s|^{4-\kappa}),
	\qquad s\in\mathbb{R},\quad \kappa\in(0,4].
	$$
	By employing Strichartz estimates adapted to bounded domains, the authors
	established the existence of global Shatah--Struwe solutions and proved that
	the dynamical system generated by these solutions admits a global attractor. The non-autonomous counterpart of problem~\eqref{model_averaged_damping-1},
	with a time-dependent external force $g=g(x,t)$, was later investigated by
	Zhao \textit{et al.}~\cite{YZZT}. In that setting, they introduced the concept
	of a uniform $\varphi$-attractor and showed that the corresponding family of
	Shatah--Struwe solution processes possesses a uniform polynomial attractor.
	
	Subsequently, Zhou \textit{et al.}~\cite{ZLZ} investigated the energy-critical
	quintic wave equation with a nonlocal averaged damping mechanism, namely,
	problem~\eqref{model_averaged_damping-1}, where the nonlinear dissipation is
	given by
	$
	\|u_t(t)\|^{q}u_t(t).
	$
	In this critical setting, the authors established the existence and described
	the structure of the weak, strong, and exponential attractors associated with
	the dynamical system generated by the problem. Moreover, the monotone character of the dissipative coefficient played a fundamental role in the application of Strichartz estimates to overcome the difficulties arising from the critical growth of the source term.
	
	A particular class of averaged damping mechanisms is given by nonlinear
	nonlocal Kelvin--Voigt damping, which arises naturally in models where the
	dissipative coefficient depends on a global quantity associated with the
	velocity field. Such mechanisms, involving terms of the form
	$
	-\|\nabla u_t(t)\|^{\alpha}\Delta u_t,
	$
	were considered in plate models by Aloui \textit{et al.}
	~\cite[Section~8, Eqs.~(8.2), (8.3), (8.11), and (8.12)]{ABH}. More recently, Cavalcanti \textit{et al.}
	~\cite{Cavalcanti-Cavalcanti-Faria-Okawa} investigated this type of nonlinear
	nonlocal Kelvin--Voigt damping in the homogeneous wave equation
	\begin{align}\label{model-cavalcanti}
		u_{tt}-\Delta u-\|\nabla u_t(t)\|^{2}\Delta u_t=0,
	\end{align}
	posed in a bounded domain with homogeneous Dirichlet boundary conditions.
	They established the well-posedness of the corresponding initial-boundary value
	problem and derived decay estimates for the total energy, describing the
	asymptotic stability of the system.
	
	This is where the present work comes in. Our problem \eqref{P} extends
	the model \eqref{model-cavalcanti} and the analysis developed by
	Cavalcanti \textit{et al.}~\cite{Cavalcanti-Cavalcanti-Faria-Okawa} by
	incorporating a nonlocal nonlinear Kelvin--Voigt damping term of the
	form $-\|\nabla u_t\|_{L^2(\Omega)}^\alpha\Delta u_t$, with
	$\alpha\in\mathbb{R}_+$, together with an energy-critical quintic source
	term and a nontrivial external force $h$. While Cavalcanti
	\textit{et al.}~\cite{Cavalcanti-Cavalcanti-Faria-Okawa} established
	well-posedness and energy decay for the related model
	\eqref{model-cavalcanti}, we extend their analysis by investigating the
	asymptotic dynamics generated by the associated dynamical system over
	the full range $\alpha\in\mathbb{R}_+$. In particular, we establish the
	existence and properties of compact global attractors and analyze their
	dependence on the damping parameter. A key feature of our approach is
	that the well-posedness theory exploits the regularizing effect of the
	nonlinear Kelvin--Voigt damping, together with appropriate regularity
	and dissipativity assumptions on the source term, rather than relying
	on Strichartz estimates.
	
	\subsection{Organization of the paper}
	
	The paper is organized as follows. Section 2 introduces the functional
	setting and the assumptions used throughout the analysis. Section 3 is
	devoted to the well-posedness of problem \eqref{P}, with the main result
	stated in Theorem \ref{theo-global-solution}. In Section 4, we establish the
	existence of a family of compact global attractors for the associated
	dynamical systems; see Theorem \ref{theo-main}. Section 5 concerns the
	Kolmogorov $\varepsilon$-entropy of these attractors. We derive the
	entropy estimate in Theorem \ref{theo-main2} and, in the particular case
	$\alpha=0$, obtain the finite fractal dimensionality of the global
	attractor. Section 6 is devoted to the higher regularity of the
	attractors in the space $\mathcal{H}_1$, as stated in Theorem
	\ref{theo-regularity}. Finally, Section 7 addresses the upper
	semicontinuity of the family of attractors at $\alpha=0$; see Theorem
	\ref{upper-semicontinuity}.
	
	\section{Functional Setting, Assumptions and Main Results}
	\subsection{Functional Setting and the Assumptions Imposed}
	
	We first introduce some notations that will be used throughout the paper.
	We denote by $H^s(\Omega)$ the $L^2(\Omega)$-based Sobolev space of order
	$s$, and by $H_0^s(\Omega)$ the closure of $C_0^\infty(\Omega)$ in
	$H^s(\Omega)$. Throughout the paper, we use the notation
	$$
	\|u\|:=\|u\|_{L^2(\Omega)},
	\qquad
	\|u\|_p:=\|u\|_{L^p(\Omega)},
	\qquad
	(u,v):=(u,v)_{L^2(\Omega)}.
	$$
	The inner product and norm in $H_0^1(\Omega)$ are given by
	$$
	((u,v)):=(\nabla u,\nabla v),
	\qquad
	\|u\|_{H_0^1(\Omega)}:=\|\nabla u\|.
	$$
	We consider the Laplace operator $-\Delta$ associated with the bilinear form
	$$
	a(u,v):=(\nabla u,\nabla v),
	\qquad
	u,v\in H_0^1(\Omega).
	$$
	The analysis is carried out in the phase space
	$$
	\mathcal H:=H_0^1(\Omega)\times L^2(\Omega),
	$$
	which is the natural space for weak solutions of problem~\eqref{P}. It is
	endowed with the norm
	\[
	\|U\|_{\mathcal H}^2
	:=
	\|\nabla u\|^2+\|v\|^2,
	\qquad
	U=(u,v).
	\]
	For the construction of strong solutions, we also introduce the more regular
	space
	$$
	\mathcal H_1
	:=
	\bigl(H^2(\Omega)\cap H_0^1(\Omega)\bigr)
	\times H_0^1(\Omega),
	$$
	equipped with the norm
	$$
	\|U\|_{\mathcal H_1}^2
	:=
	\|\Delta u\|^2+\|\nabla v\|^2,
	\qquad
	U=(u,v).
	$$
	The energy functional associated with problem~\eqref{P} is defined by
	\begin{align}\label{energy-functional}
		E_U(t)
		=
		\frac12\|U(t)\|_{\mathcal H}^2
		+\int_\Omega F(u)\,dx
		-\int_\Omega h u\,dx,
	\end{align}
	where $F(u):=\int_0^uf(\tau)d\tau$. 
	
	We shall assume throughout the paper that the external force $h$ and the nonlinear term $f$ satisfy the following conditions.
	\begin{assumption}\rm\label{Assumption}
		Let $h\in L^2(\Omega)$ be a given external force. Assume that
		$f\in C^2(\mathbb{R})$, $f(0)=0$, and that there exists a positive
		constant $L_f>0$ such that
		\begin{equation}
			\label{hyp_f''}
			|f''(s)|\le L_f\bigl(1+|s|^{3}\bigr),
			\qquad \forall s\in\mathbb{R}.
		\end{equation}
		Moreover, the following standard dissipativity condition holds:
		\begin{equation}
			\label{hyp-inf-f}
			\lim_{|s|\to\infty}\frac{f(s)}{s}>-\lambda_1,
		\end{equation}
		where $\lambda_1$ denotes the first eigenvalue of the Laplacian
		$-\Delta$ subject to homogeneous Dirichlet boundary conditions.
	\end{assumption}
	
	\begin{remark}\rm From Mean Value Theorem and \eqref{hyp_f''}, we get
		\begin{eqnarray}\label{hyp_f'}
			|f'(s)|\le C_f(1+|s|^4),\quad \forall s\in \mathbb{R},
		\end{eqnarray}
		for some constant $C_{f}>0$.
		Condition \eqref{hyp-inf-f} implies that there exists a constant $\nu\in [0,\lambda_1)$ and $C_{\nu}>0$ such that $F(s):=\int_0^sf(\tau)d\tau$ satisfies:
		\begin{eqnarray}\label{2.4}
			-C_{\nu}-\frac{\nu}{2}|s|^2\le F(s)\le f(s)s+\frac{\nu}{2}|s|^2,\quad \forall s\in \mathbb{R},
		\end{eqnarray}
	\end{remark}
	\begin{remark}\label{lowerbounded}\rm
		From the conditions \eqref{hyp_f'} and \eqref{hyp-inf-f}, it follows that
		\begin{equation}\label{lower1}
			f'(s)\geq -K_f, \qquad s\in\mathbb{R},
		\end{equation}
		for some positive constant $K_f$. Indeed, by the dissipativity condition \eqref{hyp-inf-f}, there exist $\mu \in [0,\lambda_1)$ and a sufficiently large $N>0$ such that
		\begin{equation*}
			f'(s) \ge -\mu, \quad \text{for } |s|>N.
		\end{equation*}
		On the other hand, from \eqref{hyp_f'}, we have
		\begin{equation*}
			|f'(s)| \le C_{f}(1+|N|^{4}), \quad \text{for } |s| \le N.
		\end{equation*}
		Taking $K_f = \max\{\mu, C_{f}(1+|N|^{4})\}$ yields \eqref{lower1}.
	\end{remark}
	
	Throughout the article we will use $0<\omega\le 1$ as the constant defined by
	\begin{align}
		\label{def-omega}\omega:=1-\frac{\nu}{\lambda_1}>0.
	\end{align}
	
	\section{Well-posedness}
	We first specify the notions of strong and weak solutions used in the analysis of problem \eqref{P}. The corresponding well-posedness result is stated in Theorem \ref{theo-global-solution}.
	\begin{definition}\label{3.1}
		Let $(u_0,u_1)\in\mathcal H$. A function $u$ with
		$(u(0),u_t(0))=(u_0,u_1)$ is called a solution of problem \eqref{P}
		on $[0,T]$ according to the following definitions:
		\begin{itemize}
			\item[{\bf (S)}] $u$ is a \textbf{strong solution} if
			$(u,u_t)\in L^\infty(0,T;\mathcal H_1)$ and, for every
			$\psi\in H_0^1(\Omega)$,
			\begin{align*}
				\bigl(
				u_{tt}-\Delta u
				-\|\nabla u_t\|^\alpha\Delta u_t
				+f(u)-h,\psi
				\bigr)=0
			\end{align*}
			for a.e. $t\in(0,T)$.
			
			\item[{\bf (W)}] $u$ is a \textbf{weak solution} if
			$U=(u,u_t)\in L^\infty(0,T;\mathcal H)$ and, for every
			$\psi\in H_0^1(\Omega)$ and $t\in[0,T]$,
			\begin{align}\label{variational-formula}
				&(u_t(t),\psi)
				+\int_0^t
				\big[
				(\nabla u(\tau),\nabla\psi)
				+\|\nabla u_t(\tau)\|^\alpha
				(\nabla u_t(\tau),\nabla\psi)
				+(f(u(\tau)),\psi)
				\big]\,d\tau
				\nonumber\\
				&\qquad=(u_1,\psi)
				+\int_0^t(h,\psi)\,d\tau.
			\end{align}
		\end{itemize}
	\end{definition}
	
	The following theorem establishes the well-posedness of problem~\eqref{P} in the energy space $\mathcal{H}$. The existence of strong solutions is obtained by the Galerkin
	method, while weak solutions are constructed by a density argument.
	
	\begin{theorem}{\bf [Global existence]}\label{theo-global-solution}
		Let $T>0$ be arbitrary and let $\alpha\in\mathbb{R}_+$. Under Assumption
		\ref{Assumption}, the following statements hold.
		\begin{itemize}
			\item[{\bf(I)}] For every $U_0=(u_0,u_1)\in\mathcal H_1$, there exists a unique strong solution of problem \eqref{P} on $[0,T]$ such that
			\begin{align}\label{class-regular-solution}
				U=(u,u_t)\in L^\infty(0,T;\mathcal H_1),\quad
				u_t\in L^{\alpha+2}
				(0,T;
				D((-\Delta)^{\frac{\alpha+4}{2(\alpha+2)}})),\quad
				u_{tt}\in L^\infty(0,T;H^{-1}(\Omega)).
			\end{align}
			\item[{\bf(II)}] For every $U_0=(u_0,u_1)\in\mathcal H$, there exists a unique
			weak solution of problem \eqref{P} satisfying
			\begin{align*}
				U=(u,u_t)&\in C([0,T];\mathcal H),\quad
				u_t\in L^{\alpha+2}(0,T;H_0^1(\Omega)).
			\end{align*}
			\item[{\bf(III)}] Both classes of solutions depend continuously on the initial data. More
			precisely,
			\begin{align}\label{Lipschtz}
				\|U^1(t)-U^2(t)\|_{\mathcal H}^2
				\leq C_{R,T}
				\|U^1(0)-U^2(0)\|_{\mathcal H}^2,
				\qquad t\in[0,T],
			\end{align}
			where $U^j(t)=(u^j(t),u_t^j(t))$, $j=1,2$, correspond to initial data
			$U^j(0)=(u_0^j,u_1^j)$ satisfying
			$\|U^j(0)\|_{\mathcal H}\leq R$. In particular, if
			$U_0^1=U_0^2$, then $U^1(t)=U^2(t)$ for all $t\in[0,T]$.
			\item[{\bf(IV)}] Moreover, the energy equality
			\begin{equation}\label{ei}
				E_U(t)+\int_s^t\|\nabla u_t(\tau)\|^{\alpha+2}\,d\tau
				=E_U(s),
				\qquad 0\leq s\leq t,
			\end{equation}
			holds, where $E_U$ is defined in \eqref{energy-functional}.
		\end{itemize}
	\end{theorem}
	\noindent{\bf Proof of Theorem \ref{theo-global-solution}-{\bf(I)}} To prove the existence of strong solutions, we employ the standard
	Faedo--Galerkin method. 
	\subsection*{Galerkin Approximation}
	Let
	$
	0<\lambda_1\le\lambda_2\le\cdots
	$
	be the eigenvalues of the Dirichlet Laplacian $-\Delta$, and let
	$\{\omega_j\}_{j\in\mathbb N}$ be the corresponding eigenfunctions.
	Assume that the boundary
	$\Gamma=\partial\Omega$ is sufficiently smooth so that
	$
	\omega_j\in H^2(\Omega)\cap H_0^1(\Omega),$ $j\in\mathbb N,
	$
	and that $\{\omega_j\}_{j\in\mathbb N}$ forms an orthonormal basis of
	$L^2(\Omega)$. For each $n\in\mathbb N$, define
	$
	V_n=\operatorname{span}\{\omega_1,\ldots,\omega_n\},
	$
	and let
	$\mathcal P_n:L^2(\Omega)\to V_n$
	denote the orthogonal projection onto $V_n$.
	Choose
	$
	(u_{0n},u_{1n})\in V_n\times V_n
	$
	such that
	$$
	(u_{0n},u_{1n})\to(u_0,u_1)
	\quad\text{in }\mathcal H_1.
	$$
	We construct Galerkin approximations of the form
	$$
	u^n(t)=\sum_{j=1}^n y_{jn}(t)\omega_j,
	\qquad
	y_{jn}(t)=(u^n(t),\omega_j),
	$$
	where the coefficients $\{y_{jn}\}_{j=1}^n$ satisfy the projected system
	\begin{align}
		\begin{split}\label{approximate-problem}
			&\bigl(
			u_{tt}^n-\Delta u^n
			-\|\nabla u_t^n(t)\|^\alpha\Delta u_t^n
			+f(u^n)-h,\omega_j
			\bigr)=0,
			\quad t>0,\quad j=1,\ldots,n,\\
			&(u^n(0),u_t^n(0))=(u_{0n},u_{1n}),
		\end{split}
	\end{align}
	By the classical theory of ordinary differential equations, the finite-dimensional problem \eqref{approximate-problem} admits a unique local solution on a maximal interval $[0,t_n)$, where $0<t_n\leq\infty$. The first a priori estimate established below shows that the approximate solution extends to the entire interval $[0,T]$. The second and third a priori estimates provide bounds that are independent of the Galerkin dimension $n$. These estimates yield the compactness and regularity required to extract convergent subsequences and pass to the limit in the approximate problem \eqref{approximate-problem}, thereby establishing the existence of a strong solution to problem \eqref{P}.
	
	\subsection*{A Priori Estimates}
	\paragraph{\it First estimate}
	We begin by deriving the standard energy estimate for the Galerkin approximations. Since
	$
	\|U_0\|_{\mathcal H}\le R,
	$
	the approximating initial data can be chosen so that
	\begin{align}\label{conv-initial-data}
		\|U^n(0)\|_{\mathcal H}\le C_R,
		\qquad
		\|U^n(0)-U_0\|_{\mathcal H}\longrightarrow0
		\quad\text{as }n\to\infty.
	\end{align}
	Fix $t\in(0,t_n)$. Taking $\omega_j=u_t^n$ in
	\eqref{approximate-problem} and integrating the resulting identity over
	$(0,t)$, we obtain
	\begin{align}\label{Est1-1}
		E_{U^n}(t)
		+\int_0^t\|\nabla u_t^n(s)\|^{\alpha+2}\,ds
		=
		E_{U^n}(0),
		\qquad t\in(0,t_n),
	\end{align}
	where $E_{U^n}$ denotes the energy functional defined in
	\eqref{energy-functional}, evaluated along the Galerkin approximations. By the assumption \eqref{2.4} and the Poincar\'e inequality, we have
	\begin{align*}
		\int_{\Omega}F(u^n)\,dx
		\geq
		-\frac{C_{\nu}}{2}\|u^n(t)\|^2
		-C_{\nu}|\Omega| \geq
		-\frac{C_{\nu}}{2\lambda_1}
		\|\nabla u^n(t)\|^2
		-C_{\nu}|\Omega|.
	\end{align*}
	On the other hand, by Hölder's and Young's inequalities, we obtain
	\begin{align*}
		\int_{\Omega}h u^n\,dx
		&\leq
		\|h\|\|u^n\|
		\leq
		\frac{1}{\lambda\sigma}\|h\|^2
		+\frac{\sigma}{4}\|\nabla u^n\|^2,
		\qquad \sigma>0.
	\end{align*}
	Choosing $\sigma=\omega$, where $\omega$ is defined in \eqref{def-omega},
	and combining the above estimates with the definition of
	$E_{U^n}$, we obtain
	\begin{align}\label{Est1-2}
		E_{U^n}(t)
		&\geq
		\frac{1}{2}\|u_t^n(t)\|^2
		+\frac{\omega}{2}\|\nabla u^n(t)\|^2
		-\frac{\omega}{4}\|\nabla u^n(t)\|^2
		-\frac{1}{\omega\lambda_1}\|h\|^2
		-C_\nu|\Omega| \notag\\
		&\geq
		\frac{\omega}{4}\|U^n(t)\|_{\mathcal H}^{2}
		-C_\nu|\Omega|
		-\frac{1}{\omega\lambda_1}\|h\|^2 .
	\end{align}
	Moreover, the growth assumption on $F$ in \eqref{2.4}, combined with the Sobolev embedding $H_0^1(\Omega)\hookrightarrow L^6(\Omega)$, yields
	\begin{align}\begin{split}\label{Est1-3}
			E_{U^n}(0)
			\leq C\left(\|U_0^n\|_{\mathcal H}^{2}
			+\|u^n_0\|^2+\|u_0^n\|_{6}^{6}
			+\|h\|\|u^n_0\|\right)
			\leq c_R.
	\end{split}\end{align}
	Consequently, by combining the estimate \eqref{Est1-2}, the energy
	identity \eqref{Est1-1}, and the uniform bound \eqref{Est1-3} with the
	convergence property \eqref{conv-initial-data}, we obtain
	\begin{align}\label{Est-1}
		\|U^n(t)\|_{\mathcal H}^{2}
		+\int_0^t\|\nabla u_t^n(s)\|^{\alpha+2}\,ds
		\leq \frac{4}{\omega}\left[C_{\nu}|\Omega|+\frac{1}{\lambda_1\omega}\|h\|^2+c_R\right]=:C_R,\quad \forall t\in[0,t_n).
	\end{align}
	Moreover, the constant $C_R$ is independent of both $n$ and $t_n$.
	Therefore, the local existence time for the Galerkin approximations can be
	chosen uniformly with respect to $n$. In particular, the standard continuation
	argument implies that
	$t_n=T$
	for any fixed $T>0$, and hence the estimate \eqref{Est-1} holds on the whole
	interval $[0,T]$.
	
	\medskip
	\paragraph{\it Second estimate}
	We now consider the case where $U_0\in\mathcal{H}_1$. The approximating initial data can be chosen such that
	\begin{align}\label{conv-initial-data-2}
		\|U^n(0)\|_{\mathcal{H}_1}\le R,
		\qquad
		\|U^n(0)-U_0\|_{\mathcal{H}_1}\longrightarrow0
		\quad\text{as }n\to\infty.
	\end{align}
	Setting $\omega_j=-\Delta u_t^n$ in \eqref{approximate-problem}, we obtain
	\begin{align}
		\label{Est3-a}
		\begin{split}
			&\frac{d}{dt}\left[
			\frac12\|U^n(t)\|_{\mathcal H_1}^2
			+\frac12\int_\Omega f'(u^n)|\nabla u^n|^2\,dx
			-\int_\Omega h\Delta u^n\,dx
			\right]+\|\nabla u_t^n(t)\|^\alpha
			\|\Delta u_t^n(t)\|^2\\
			&\quad=
			\frac12
			\int_\Omega
			f''(u^n)u_t^n|\nabla u^n|^2\,dx.
		\end{split}
	\end{align}
	We introduce the energy functional
	\begin{align*}
		E^1_{U^n}(t):=
		\frac12\|U^n(t)\|_{\mathcal H_1}^2
		+\frac12\int_\Omega f'(u^n)|\nabla u^n|^2\,dx
		-\int_\Omega h\Delta u^n\,dx
		+\frac{K_f}{2}\|\nabla u^n(t)\|^2
		+\|h\|^2.
	\end{align*}
	Hence, \eqref{Est3-a} can be rewritten as
	\begin{align}
		\label{Est3-c}
		\begin{split}
			\frac{d}{dt}E^1_{U^n}(t)
			&+\|\nabla u_t^n(t)\|^\alpha
			\|\Delta u_t^n(t)\|^2
			\\
			&=
			K_f\int_\Omega
			\nabla u_t^n\cdot\nabla u^n\,dx
			+\frac12
			\int_\Omega
			f''(u^n)u_t^n|\nabla u^n|^2\,dx.
		\end{split}
	\end{align}
	We next show that $E^1_{U^n}(t)$ is equivalent to
	$\|U^n(t)\|_{\mathcal H_1}^2$.
	Indeed, by H\"older's inequality,
	$$
	\left|
	\int_\Omega h\Delta u^n\,dx
	\right|
	\le
	\|h\|^2+\frac14\|\Delta u^n(t)\|^2.
	$$
	Moreover, by Assumption \eqref{hyp_f'}, Remark \ref{lowerbounded},
	H\"older's inequality with
	$\frac23+\frac13=1$,
	and the embedding
	$H_0^1(\Omega)\hookrightarrow L^6(\Omega)$,
	\begin{align*}
		-\frac{K_f}{2}\|\nabla u^n(t)\|^2
		\le
		\frac12
		\int_\Omega
		f'(u^n)|\nabla u^n|^2\,dx
		\le&\;
		C
		\int_\Omega
		(1+|u^n|^4)|\nabla u^n|^2\,dx
		\\
		\le&\;
		C\left(1+
		\|u^n(t)\|_6^4\right)
		\|\nabla u^n(t)\|_6^2
		\\
		\le&\;
		C_R\|\Delta u^n(t)\|^2.
	\end{align*}
	Combining the above two estimates with the definition of
	$E^1_{U^n}$, we obtain
	\begin{align}
		\label{equiv_E1}
		\frac14\|U^n(t)\|_{\mathcal H_1}^2
		\le
		E^1_{U^n}(t)
		\le
		C_R\|U^n(t)\|_{\mathcal H_1}^2.
	\end{align}
	Now, by the Poincaré inequality, estimate \eqref{Est-1}, and the equivalence
	\eqref{equiv_E1}, we have
	\begin{align*}
		K_f\int_\Omega \nabla u_t^n\cdot\nabla u^n\,dx
		\le
		\frac{K_f}{\lambda_1^{1/2}}
		\|\nabla u_t^n(t)\|
		\|\Delta u^n(t)\|
		\le
		C_R\|\Delta u^n(t)\|\le
		C_R E_{U^n}^1(t).
	\end{align*}
	Moreover, by Assumption \eqref{hyp_f''}, Hölder's inequality with
	$\frac12+\frac16+\frac13=1$, the embedding
	$H_0^1(\Omega)\hookrightarrow L^6(\Omega)$, estimate
	\eqref{Est-1}, and \eqref{equiv_E1}, we obtain
	\begin{align*}
		\frac12
		\int_\Omega
		f''(u^n)u_t^n|\nabla u^n|^2\,dx
		&\le
		C\bigl(1+\|u^n(t)\|_6^3\bigr)
		\|u_t^n(t)\|_6
		\|\nabla u^n(t)\|_6^2
		\\
		&\le
		C\bigl(1+\|\nabla u^n(t)\|^3\bigr)
		\|\nabla u_t^n(t)\|
		\|\Delta u^n(t)\|^2
		\\
		&\le
		C_R
		\bigl(1+\|\nabla u_t^n(t)\|^{\alpha+2}\bigr)
		\|\Delta u^n(t)\|^2
		\\
		&\le
		C_R
		\bigl(1+\|\nabla u_t^n(t)\|^{\alpha+2}\bigr)
		E_{U^n}^1(t).
	\end{align*}
	Next, by the interpolation inequality associated with the Dirichlet Laplacian,
	\begin{align*}
		\|\nabla u_t^n\|^{\alpha}\|\Delta u_t^n\|^2
		=
		\left[
		\|(-\Delta)^{1/2}u_t^n\|^{\frac{\alpha}{\alpha+2}}
		\|(-\Delta)u_t^n\|^{\frac{2}{\alpha+2}}
		\right]^{\alpha+2}
		\ge
		C_\alpha
		\|
		(-\Delta)^{\frac{\alpha+4}{2(\alpha+2)}}
		u_t^n
		\|^{\alpha+2}.
	\end{align*}
	Hence, combining the last three estimates in \eqref{Est3-c}, we conclude that
	\begin{align}
		\label{Est3-e}
		\frac{d}{dt}E_{U^n}^1(t)
		+
		C_\alpha
		\|
		(-\Delta)^{\frac{\alpha+4}{2(\alpha+2)}}
		u_t^n(t)
		\|^{\alpha+2}
		\le
		C_R
		\bigl(1+\|\nabla u_t^n(t)\|^{\alpha+2}\bigr)
		E_{U^n}^1(t).
	\end{align}
	Integrating \eqref{Est3-e} over $[0,t]$ and using estimates
	\eqref{Est-1}, \eqref{equiv_E1}, and assumption
	\eqref{conv-initial-data-2}, we obtain
	$$
	E_{U^n}^1(t)
	+
	C_\alpha
	\int_0^t
	\|
	(-\Delta)^{\frac{\alpha+4}{2(\alpha+2)}}
	u_t^n(s)\|^{\alpha+2}\,ds
	\le C_R.
	$$
	Finally, using the equivalence of norms established in
	\eqref{equiv_E1}, we conclude that
	\begin{align}
		\label{strongy-final-2}
		\|U^n(t)\|_{\mathcal H_1}^{2}
		+
		\int_0^t
		\|
		(-\Delta)^{\frac{\alpha+4}{2(\alpha+2)}}
		u_t^n(s)\|^{\alpha+2}\,ds
		\le C_R,
		\qquad
		t\in[0,T].
	\end{align}

	\medskip
	\paragraph{\it Third estimate}
	By Eq.~\eqref{P} and the continuous embedding
	$L^2(\Omega)\hookrightarrow H^{-1}(\Omega)$, we obtain
	\begin{align*}
		\|(-\Delta)^{-1/2}u_{tt}^n(t)\|
		\le
		C\left[\,
		\|\nabla u^n(t)\|
		+\|\nabla u_t^n(t)\|^{\alpha+1}
		+\|f(u^n(t))\|
		+\|h\|
		\,\right].
	\end{align*}
	Moreover, by the growth assumption \eqref{hyp_f'}, the Sobolev embedding
	$H^2(\Omega)\hookrightarrow L^\infty(\Omega)$, the embedding
	$L^\infty(\Omega)\hookrightarrow L^{10}(\Omega)$, and the elliptic estimate
	$\|u^n\|_{H^2}\le C\|\Delta u^n\|$, we obtain
	\begin{align*}
		\|f(u^n)\|
		\le
		C\left(1+\|u^n\|_{{10}}^5\right)\le
		C\left(1+\|u^n\|_{H^2}^5\right)\le
		C\left(1+\|\Delta u^n\|^5\right)
		\le
		C_R.
	\end{align*}
	Consequently,
	\begin{align}\label{Est3-1}
		\|(-\Delta)^{-1/2}u_{tt}^n(t)\|
		\le C_R,
		\qquad
		\forall\, t\in[0,T].
	\end{align}
	
	\subsection*{Passage to the Limit}
	By virtue of the uniform estimates \eqref{Est-1}, \eqref{strongy-final-2},
	and \eqref{Est3-1} we can extract a subsequence (still denoted by $n$) such that
	\begin{align}
		\left\{
		\begin{aligned}
			U^n=(u^n,u_t^n)
			&\rightharpoonup^\ast U=(u,u_t)
			&&\text{in }L^\infty(0,T;\mathcal H_1),\\
			u_{tt}^n
			&\rightharpoonup^\ast u_{tt}
			&&\text{in }L^\infty(0,T;H^{-1}(\Omega)),\\
			u_t^n
			&\rightharpoonup u_t
			&&\text{in }L^{\alpha+2}\bigl(0,T;
			D((-\Delta)^{\frac{\alpha+4}{2(\alpha+2)}})\bigr),\\
			f(u^n)
			&\rightharpoonup f(u)
			&&\text{in }L^2(0,T;L^2(\Omega)).
		\end{aligned}
		\right.
		\label{conv-1}
	\end{align}
	By the Aubin--Lions compactness lemma \cite{Lions},
	applied to $\{u^n\}$ with
	$
	H^2(\Omega)\cap H_0^1(\Omega)
	\hookrightarrow\hookrightarrow
	H_0^1(\Omega),
	$
	and to $\{u_t^n\}$ with
	$
	H_0^1(\Omega)
	\hookrightarrow\hookrightarrow
	L^2(\Omega)
	\hookrightarrow
	H^{-1}(\Omega),
	$
	we conclude that
	\begin{align*}
		\left\{
		\begin{aligned}
			u^n&\to u
			&&\text{strongly in } C([0,T];H_0^1(\Omega)),\\
			u_t^n&\to u_t
			&&\text{strongly in } C([0,T];L^2(\Omega)).
		\end{aligned}
		\right.
	\end{align*}
	Consequently,
	\begin{align*}
		U^n\to U
		\quad\text{strongly in }C([0,T];\mathcal H).
	\end{align*}
	Define the nonlinear operator
	$$
	B:H_0^1(\Omega)\to H^{-1}(\Omega)
	$$
	by
	$$
	\langle B(u_t),v_t\rangle
	=
	\|\nabla u_t\|^\alpha(\nabla u_t,\nabla v_t),
	\qquad
	u_t,v_t\in H_0^1(\Omega).
	$$
	Clearly, $B$ is hemicontinuous and
	\begin{align}\label{3.12}
		\|B(u_t)\|_{H^{-1}}
		&=
		\sup_{\|\nabla v_t\|\leq 1}
		\|\nabla u_t\|^\alpha
		(\nabla u_t,\nabla v_t) \notag\\
		&\leq
		\sup_{\|\nabla v_t\|\leq 1}
		\|\nabla u_t\|^\alpha
		\|\nabla u_t\|
		\|\nabla v_t\|
		\leq
		\|\nabla u_t\|^{\alpha+1}.
	\end{align}
	Hence, $B$ is bounded. Moreover, using the standard monotonicity property of
	the mapping $z\mapsto |z|^\alpha z$, there exists $C_\alpha>0$ such that
	\[
	\langle B(u_t)-B(v_t),u_t-v_t\rangle
	\geq
	C_\alpha
	\|\nabla(u_t-v_t)\|^{\alpha+2},
	\]
	for all $u_t,v_t\in H_0^1(\Omega)$. Therefore, $B$ is strongly monotone. Consequently, by the monotone operator theory (see, e.g.,
	Corollary 2.3 in \cite{Showalter}), we have
	$$
	B(u_t^n)\rightharpoonup B(u_t)
	\quad\text{in }H^{-1}(\Omega),
	\quad\text{for a.e. }t\in(0,T).
	$$
	Furthermore, by \eqref{3.12} and \eqref{Est-1}, we have
	$$
	\int_0^T
	\|B(u_t^n)\|_{H^{-1}}^{\frac{\alpha+2}{\alpha+1}}\,dt
	\leq
	\int_0^T
	\|\nabla u_t^n\|^{\alpha+2}\,dt
	\leq C_R.
	$$
	Hence,
	\begin{align}\label{conv-3}
		B(u_t^n)\rightharpoonup B(u_t)
		\quad\text{in}\quad
		L^{\frac{\alpha+2}{\alpha+1}}(0,T;H^{-1}(\Omega)).
	\end{align}
	Passing to the limit as $n\to+\infty$ in \eqref{approximate-problem}, using \eqref{conv-1}--\eqref{conv-3}, we conclude that the limit function $u$ is a strong solution of problem \eqref{P} with the regularity
	specified in \eqref{class-regular-solution}.
	
	\subsection*{Continuous Dependence and Uniqueness of Strong Solutions}
	Let $U^j=(u^j,u_t^j)$, $j=1,2$, be strong solutions of problem
	\eqref{P} corresponding to the initial data
	$U^j(0)\in\mathcal H_1\subset\mathcal H$, with
	$||U^j(0)||_{\mathcal H}\le R$, $j=1,2$.
	Defining
	$
	w=u^1-u^2,
	$ and $
	W=(w,w_t),$
	we find that $W$ satisfies
	\begin{eqnarray}\left\{\begin{array}{l}\label{3.14}
			w_{tt}-\Delta w
			-\|\nabla u_t^1\|^\alpha\Delta u_t^1
			+\|\nabla u_t^2\|^\alpha\Delta u_t^2
			+f(u^1)-f(u^2)=0,\\
			w=0\quad \mbox{on}\quad \Gamma,\qquad
			W(0)=U^1(0)-U^2(0).
		\end{array}\right.
	\end{eqnarray}
	Taking $w_t$ as a multiplier in \eqref{3.14}, we obtain
	\begin{align}\begin{split}\label{3.15}
			&\frac{1}{2}\frac{d}{dt}
			\left[\,
			\|W\|_{\mathcal H}^2
			+
			\int_\Omega\int_0^1
			f'(\theta u^1+(1-\theta)u^2)
			\,d\theta\, |w|^2\,dx
			\,\right]
			\\
			&\quad
			-\int_{\Omega}
			\left(
			\|\nabla u_t^1\|^{\alpha}\Delta u_t^1
			-
			\|\nabla u_t^2\|^{\alpha}\Delta u_t^2
			\right)w_t\,dx
			\\
			&=
			\frac12\int_\Omega\int_0^1
			f''(\theta u^1+(1-\theta)u^2)
			(\theta u_t^1+(1-\theta)u_t^2)
			\,d\theta\, |w|^2\,dx.
	\end{split}\end{align}
	Note that
	\begin{align*}
		-\int_{\Omega}
		\left(
		\|\nabla u_t^1\|^{\alpha}\Delta u_t^1
		-
		\|\nabla u_t^2\|^{\alpha}\Delta u_t^2
		\right)
		w_t dx
		&=
		-\frac12
		\left(
		\|\nabla u_t^1\|^{\alpha}
		+
		\|\nabla u_t^2\|^{\alpha}
		\right)
		\int_{\Omega}\Delta w_t\,w_t dx
		\\
		&\qquad
		-\frac12
		\left(
		\|\nabla u_t^1\|^{\alpha}
		-
		\|\nabla u_t^2\|^{\alpha}
		\right)
		\int_{\Omega}\Delta(u_t^1+u_t^2)w_t dx
		\\
		&=
		\frac12
		\left(
		\|\nabla u_t^1\|^{\alpha}
		+
		\|\nabla u_t^2\|^{\alpha}
		\right)
		\|\nabla w_t\|^2
		\\
		&\qquad
		+\frac12
		\left(
		\|\nabla u_t^1\|^{\alpha}
		-
		\|\nabla u_t^2\|^{\alpha}
		\right)
		\left(
		\|\nabla u_t^1\|^2
		-
		\|\nabla u_t^2\|^2
		\right)\ge 0.
	\end{align*}
	Moreover, using assumption \eqref{hyp_f''}, Hölder's inequality with
	$\frac12+\frac16+\frac13=1$, and the embedding
	$H_0^1(\Omega)\hookrightarrow L^6(\Omega)$, we obtain
	
	\begin{align*}
		&\frac12\int_\Omega\int_0^1
		f''(\theta u^1+(1-\theta)u^2)
		(\theta u_t^1+(1-\theta)u_t^2)
		\,d\theta\, |w|^2 dx
		\\
		\leq&
		C(1+\|u^1\|_{6}^3+\|u^2\|_{6}^3)
		(\|u_t^1\|_{6}+\|u_t^2\|_{6})
		\|w\|_{6}^2
		\\
		\leq
		&C(1+\|\nabla u^1\|^3+\|\nabla u^2\|^3)
		(\|\nabla u_t^1\|+\|\nabla u_t^2\|)
		\|\nabla w\|^2\leq
		C_R\psi(t)
		\|\nabla w\|^2,
	\end{align*}
	where $\psi(t):=(1+\|\nabla u_t^1\|^{\alpha+2}
	+\|\nabla u_t^2\|^{\alpha+2})$.
	Inserting above two estimates into \eqref{3.15} yields
	\begin{align}\label{3.16}
		\frac{d}{dt} \left[\,||W||^2_{\mathcal{H}}+ \int_\Omega\int_0^1f'(\theta u^1+(1-\theta)u^2)d\theta |w|^2 dx\right]\leq  C_R\psi(t)||W||^2_{\mathcal{H}}.
	\end{align}
	Now, we define
	\begin{align*}
		\mathcal{E}_W(t)=||W||^2_{\mathcal{H}}+ \int_\Omega\int_0^1f'(\theta u^1+(1-\theta)u^2)d\theta |w|^2 dx+K_f \| w\|^2.
	\end{align*}
	From Assumption \eqref{hyp_f'}, Remark \ref{lowerbounded},
	H\"older's inequality with
	$\frac23+\frac13=1$,
	and the embedding
	$H_0^1(\Omega)\hookrightarrow L^6(\Omega)$, we obtain
	\begin{align*}
		-\frac{K_f}{2}\|w\|^2
		\le
		\frac12
		\int_\Omega\int_0^1
		f'(\theta u^1+(1-\theta)u^2)\,d\theta\, |w|^2\,dx
		\leq
		C(1+\|u^1\|_{6}^4+\|u^2\|_{6}^4)\|w\|_{6}^2.
	\end{align*}
	Hence, using Poincaré's inequality and the embedding
	$H_0^1(\Omega)\hookrightarrow L^6(\Omega)$, we obtain
	\begin{align}\label{3.21}
		||W||_{\mathcal H}^2
		\leq
		\mathcal E_W(t)
		\leq
		C_R||W||_{\mathcal H}^2.
	\end{align}
	Therefore, by \eqref{3.21}, it follows from \eqref{3.16} that
	\begin{align}\begin{split}\label{3.22}
			\frac{d}{dt}\mathcal E_W(t)
			\leq
			C_R\psi(t)\mathcal E_W(t).
	\end{split}\end{align}
	Applying Gronwall's inequality to \eqref{3.22} and then using the equivalence of norms in \eqref{3.21}, we obtain
	\begin{align}\label{Lipschtz-1}
		||W(t)||_{\mathcal H}^2
		\leq
		C_{R,T}\,||W(0)||_{\mathcal H}^2,
		\qquad \forall\, t\in[0,T].
	\end{align}
	From \eqref{Lipschtz-1}, we immediately obtain the Lipschitz stability
	estimate \eqref{Lipschtz}. In particular, taking
	$U^1(0)=U^2(0)$ yields the uniqueness of strong solutions. This completes the proof of Theorem \ref{theo-global-solution} \textbf{(I)}.
	\qed
	
	\subsection*{Proof of Theorem \ref{theo-global-solution}-{\bf(II)}}
	The weak solution is obtained as the limit of strong solutions
	corresponding to initial data approximating $U_0$ in $\mathcal H$. Indeed,
	for each $U_0=(u_0,u_1)\in\mathcal H$, by the density of $\mathcal H_1$
	in $\mathcal H$, there exists a sequence
	$\{U_0^n=(u_0^n,u_1^n)\}\subset\mathcal H_1$ such that
	\begin{align*}
		U^n(0)=(u_0^n,u_1^n)\to U_0=(u_0,u_1)
		\quad\text{in }\mathcal H.
	\end{align*}
	By item \textbf{(I)} of Theorem \ref{theo-global-solution}, for each
	$n$, problem \eqref{P} admits a unique strong solution $U^n(t)$
	corresponding to the initial data $U^n(0)$. Moreover, by the Lipschitz
	estimate \eqref{Lipschtz},
	\begin{align*}
		\max_{t\in[0,T]}\|U^n(t)-U^m(t)\|_{\mathcal H}
		\leq
		C\bigl(\|U^n_0\|_{\mathcal H},\|U^m_0\|_{\mathcal H},T\bigr)
		\|U^n_0-U^m_0\|_{\mathcal H}
		\longrightarrow 0
	\end{align*}
	as $n,m\to\infty$. Hence, $\{U^n\}$ is a Cauchy sequence in
	$C([0,T];\mathcal H)$. Therefore, there exists
	$U\in C([0,T];\mathcal H)$ such that
	\begin{align*}
		\lim_{n\to\infty}
		\max_{t\in[0,T]}\|U^n(t)-U(t)\|_{\mathcal H}=0.
	\end{align*}
	In particular,
	\begin{align*}
		u^n\to u \quad\text{a.e. in }Q_T,
		\qquad
		f(u^n)\to f(u)\quad\text{a.e. in }Q_T,
	\end{align*}
	where
	$
	Q_T:=\Omega\times(0,T).
	$
	From assumption \ref{Assumption} and the embeddings
	$H_0^1(\Omega)\hookrightarrow L^6(\Omega)$, we obtain
	\begin{align}\label{3.27}
		\int_0^T\int_\Omega |f(u^n)|^{6/5}\,dx\,dt
		\leq
		C\int_0^T\int_\Omega
		\left(1+|u^n|^6\right)\,dx\,dt
		\leq C_T.
	\end{align}
	Thus,
	\begin{align}\label{3.28}
		\{f(u^n)\}_{n\in\mathbb N}
		\quad\text{is bounded in }L^{6/5}(Q_T).
	\end{align}
	Consequently, by \eqref{3.27} and \eqref{3.28},
	\begin{align*}
		f(u^n)\rightharpoonup f(u)
		\quad\text{weakly in }L^{6/5}(Q_T).
	\end{align*}
	Therefore,
	\begin{align}\label{3.30}
		\int_0^T\int_\Omega f(u^n)\psi\,dx\,dt
		\longrightarrow
		\int_0^T\int_\Omega f(u)\psi\,dx\,dt,
		\qquad
		\forall\,\psi\in L^6(Q_T).
	\end{align}
	Since $U^n$ is a strong solution of \eqref{P}, for every
	\begin{align*}
		\psi\in C([0,T];H_0^1(\Omega))
		\cap C^1([0,T];L^2(\Omega))
		\cap L^6(Q_T),
	\end{align*}
	we have
	\begin{align*}
		(u_t^n(t),\psi(t))
		&+\int_0^t
		\Big[
		(\nabla u^n,\nabla\psi)
		+(f(u^n),\psi)
		+\|\nabla u_t^n\|^\alpha
		(\nabla u_t^n,\nabla\psi)
		\Big]\,d\tau\\
		&=(u_1^n,\psi(0))
		+\int_0^t(h,\psi)\,d\tau,
		\qquad t\in[0,T].
	\end{align*}
	Letting $n\to\infty$, we conclude that the limiting function $u$ is a weak solution of problem
	\eqref{P}. It remains to prove uniqueness of weak solutions. Let
	\[
	U_0^1=(u_0^1,u_1^1),\qquad
	U_0^2=(u_0^2,u_1^2)\in\mathcal H
	\]
	be arbitrary initial data. Since $\mathcal H_1$ is dense in $\mathcal H$,
	there exist sequences
	\[
	U_0^{i,n}=(u_0^{i,n},u_1^{i,n})\in\mathcal H_1,
	\qquad i=1,2,
	\]
	such that
	\begin{align}\label{3.29}
		(U_0^{1,n},U_0^{2,n})
		\longrightarrow
		(U_0^1,U_0^2)
		\quad\text{strongly in }\mathcal H\times\mathcal H.
	\end{align}
	Let $U^{i,n}(t)$ be the corresponding strong solutions. By the convergence
	established above,
	\begin{align*}
		(U^{1,n},U^{2,n})
		\longrightarrow
		(U^1,U^2)
		\quad\text{strongly in }
		C([0,T];\mathcal H\times\mathcal H),
	\end{align*}
	where $U^i(t)$ is the weak solution corresponding to $U_0^i$.
	
	Since the Lipschitz estimate \eqref{Lipschtz} holds for strong solutions, we have
	\begin{align}\label{3.31}
		\|U^{1,n}(t)-U^{2,n}(t)\|_{\mathcal H}^2
		\leq
		C\bigl(\|U_0^{1,n}\|_{\mathcal H},
		\|U_0^{2,n}\|_{\mathcal H},T\bigr)
		\|U_0^{1,n}-U_0^{2,n}\|_{\mathcal H}^2,
		\qquad t\in[0,T].
	\end{align}
	Since the sequences of initial data are bounded in $\mathcal H$, the constant
	on the right-hand side of \eqref{3.31} can be chosen independently of $n$.
	Passing to the limit as $n\to\infty$ in \eqref{3.31}, and using
	\eqref{3.29}--\eqref{3.30}, we obtain
	\begin{align*}
		\|U^1(t)-U^2(t)\|_{\mathcal H}^2
		\leq
		C\bigl(\|U_0^1\|_{\mathcal H},
		\|U_0^2\|_{\mathcal H},T\bigr)
		\|U_0^1-U_0^2\|_{\mathcal H}^2,
		\qquad t\in[0,T].
	\end{align*}
	Thus, the Lipschitz estimate \eqref{Lipschtz} extends to weak solutions.
	In particular, if $U_0^1=U_0^2$, then $U^1(t)=U^2(t)$ for all
	$t\in[0,T]$. Hence, the weak solution is unique, completing the proof of Theorem \ref{theo-global-solution}-\textbf{(II)}.
	
	\subsection*{Proof of Theorem \ref{theo-global-solution}-{\bf(III)}}
	The result follows immediately from the estimate \eqref{Lipschtz-1}. 
	
	\subsection*{Proof of Theorem \ref{theo-global-solution}-{\bf(IV)}}
	We establish the energy identity by  the Friedrichs mollifier with a time cut-off. By definition \eqref{3.1}, a weak solution $u$ satisfies
	\[
	U\in L^\infty(0,T;\mathcal H),\qquad u_t\in L^{\alpha+2}(0,T;H_0^1(\Omega)),
	\]
	and for every $\psi\in C(0,T;H_0^1)\cap C^1(0,T;L^2)\cap L^6(Q_T)$,
	\begin{align}\label{ws}
		\int_0^T (u_t,\psi_t)\,dt+\int_0^T(\nabla u,\nabla\psi)\,dt
		+\int_0^T\|\nabla u_t(t)\|^\alpha(\nabla u_t,\nabla\psi)\,dt
		+\int_0^T(f(u),\psi)\,dt
		=\int_0^T(h,\psi)\,dt.
	\end{align}
	Let $\rho\in C_c^\infty(-1,1)$ be a standard mollifier with $\int_{\mathbb R}\rho=1$,  $\rho_\varepsilon(s)=\varepsilon^{-1}\rho(s/\varepsilon)$ and  $\eta_\varepsilon\in C_c^\infty(0,T)$ be a cut-off function such that 
	\begin{align*}\eta_\varepsilon(t)=
		\begin{cases}
			0, & t\in[0,\varepsilon/2]\cup[T-\varepsilon/2,T],\\
			1, & t\in[\varepsilon,T-\varepsilon].
	\end{cases}\end{align*}
	Define the test function
	\[
	\psi_\varepsilon(t):=\eta_\varepsilon(t)(\rho_\varepsilon*\widetilde{u_t})(t),
	\]
	where $\widetilde{u_t}$ denotes the zero extension of $u_t$ to $\mathbb R$. Then $\psi_\varepsilon\in C_c^\infty(0,T;H_0^1(\Omega))\subset C(0,T;H_0^1)\cap C^1(0,T;L^2)\cap L^6(Q_T)$, so it is admissible in \eqref{ws}.

	Substituting $\psi_\varepsilon$ into \eqref{ws}, we have
	\begin{align}\label{ws1}
		I_1(\varepsilon)+I_2(\varepsilon)+I_3(\varepsilon)+I_4(\varepsilon)=I_5(\varepsilon),
	\end{align}
	where the $I_j$ correspond to the five terms in \eqref{ws}.
	
	\medskip\noindent\textit{Term $I_1$.} 
	Set $\phi_\varepsilon(t) = \big(u_t(t),\psi_\varepsilon(t)\big)_{\mathcal H}$, then by the definition of $\psi_\varepsilon(t)$, we have
	\begin{align*}\psi_\varepsilon(0)=\psi_\varepsilon(T)=0, \quad \phi_\varepsilon(0)=\phi_\varepsilon(T)=0.\end{align*}
	
	Differentiating $\phi_\varepsilon(t)$ and integrating the result over $[0,T]$ yields
	\[0 = \phi_\varepsilon(T)-\phi_\varepsilon(0)=\int_0^T \langle u_{tt},\psi_\varepsilon\rangle_{H, H^{-1}}dt + \int_0^T \big(u_t,\partial_t\psi_\varepsilon\big)dt,\]
	hence,
	\[I_1(\varepsilon) = -\int_0^T \langle u_{tt},\psi_\varepsilon\rangle_{ H, H^{-1}}dt.\]
	Because $u_t\in L^{\alpha+2}(0,T;H_0^1)$ and $\eta_\varepsilon\to1$ as $\varepsilon\to0$,
	\begin{align*}
		\|\psi_\varepsilon - u_t\|_{L^{\alpha+2}(0,T;H_0^1)}
		&= \big\|\eta_\varepsilon(\rho_\varepsilon*\widetilde{u}_t) - u_t\big\|_{L^{\alpha+2}(0,T;H_0^1)}\\
		&\le \big\|\eta_\varepsilon(\rho_\varepsilon*\widetilde{u}_t - \widetilde{u}_t)\big\|_{L^{\alpha+2}(0,T;H_0^1)} + \big\|(\eta_\varepsilon-1)u_t\big\|_{L^{\alpha+2}(0,T;H_0^1)}\rightarrow0\quad \hbox{as}\quad \varepsilon\to0.
	\end{align*}
	Define \[L(w) = \int_0^T \langle u_{tt},w \rangle_{H,H^{-1}} dt,\quad w\in L^{\alpha+2}(0,T;H_0^1),\] then
	\[|L(w)| \le \|u_{tt}\|_{L^{\frac{\alpha+2}{\alpha+1}}(0,T;H^{-1})}\|w\|_{L^{\alpha+2}(0,T;H_0^1)},\]
	which means $L(w)$ is bounded linear functional on $L^{\alpha+2}(0,T;H_0^1)$. Therefore,
	\[\lim_{\varepsilon\to0}\int_0^T\langle u_{tt},\psi_\varepsilon\rangle _{H, H^{-1}}dt = \int_0^T\langle u_{tt},u_t\rangle_{H, H^{-1}}dt=\frac12\big(\|u_t(T)\|_{\mathcal H}^2 - \|u_t(0)\|_{\mathcal H}^2\big).\]
	Combining all identities,
	\begin{align*}
		\lim_{\varepsilon\to0}I_1(\varepsilon) = \frac12\big(\|u_t(0)\|_{\mathcal H}^2 - \|u_t(T)\|_{\mathcal H}^2\big).
	\end{align*}
	
	\medskip\noindent\textit{Term $I_2$.}
	Since $\nabla\psi_\varepsilon\to\nabla u_t$ in $L^2(Q_T)$, we get
	\[
	\lim_{\varepsilon\to0}I_2(\varepsilon)=\int_0^T(\nabla u,\nabla u_t)\,dt
	=\frac12\|\nabla u(T)\|^2-\frac12\|\nabla u(0)\|^2.
	\]
	
	\medskip\noindent\textit{Term $I_3$.}
	Let $\lambda(t)=\|\nabla u_t(t)\|^\alpha$. Since $u_t\in L^{\alpha+2}(0,T;H_0^1)$, we have
	\[
	\lambda\in L^{(\alpha+2)/\alpha}(0,T),\qquad \nabla u_t\in L^{\alpha+2}(Q_T).
	\]
	Moreover, $\nabla\psi_\varepsilon\to\nabla u_t$ in $L^{\alpha+2}(Q_T)$. By H\"older's inequality,
	\[
	\int_0^T \lambda(t) \bigl(\nabla u_t, \nabla \psi_\varepsilon - \nabla u_t\bigr) dt
	\le
	\left( \int_0^T \|\nabla u_t\|^{\alpha+2} dt \right)^{\frac{\alpha+1}{\alpha+2}}
	\left( \int_0^T \|\nabla \psi_\varepsilon - \nabla u_t\|^{\alpha+2} dt \right)^{\frac{1}{\alpha+2}}
	\rightarrow 0.
	\]
	Hence
	\[
	\lim_{\varepsilon\to0}I_3(\varepsilon)=\int_0^T\lambda(t)\|\nabla u_t(t)\|^2\,dt
	=\int_0^T\|\nabla u_t(t)\|^{\alpha+2}\,dt.
	\]
	
	\medskip\noindent\textit{Term $I_4$.} By assumption \eqref{hyp_f'} and Sobolev embedding $H_0^1(\Omega)\hookrightarrow L^6(\Omega)$,
	\[\|f(u(t))\|_{{6/5}} \le C\left(1+\|u(t)\|_{6}^5\right) \le C\left(1+\|\nabla u(t)\|^5\right)\le C_R,\quad \text{a.e. }t\in(0,T).\]
	Thus,
	\[\left|\int_0^T \big(f(u),\psi_\varepsilon - u_t\big)dt\right|
	\le \|f(u)\|_{L^\infty(0,T; L^{6/5})} \int_0^T \|\psi_\varepsilon - u_t\|_{L^6}dt\leq C(R,T)\|\psi_\varepsilon - u_t\|_{L^{\alpha+2}(0,T;H_0^1)}\rightarrow0 .\]
	Consequently,
	$$
	\lim_{\varepsilon\to0}I_4(\varepsilon)=\int_0^T(f(u),u_t)\,dt
	=\int_\Omega F(u(T))\,dx-\int_\Omega F(u_0)\,dx.
	$$
	
	\medskip\noindent\textit{Term $I_5$.}
	Since $h\in L^2$ and $\psi_\varepsilon\to u_t$ in $L^2(Q_T)$,
	$$
	\lim_{\varepsilon\to0}I_5(\varepsilon)=\int_0^T(h,u_t)\,dt.
	$$
	Finally, letting $\varepsilon\to0$ in \eqref{ws1} and using the estimates
	$I_1$-$I_5$, we obtain the energy identity \eqref{ei}. This completes the proof of Theorem \ref{theo-global-solution}.
	\qed

	\section{Global Attractor}
	The well-posedness result established in Theorem \ref{theo-global-solution} defines, for each $\alpha\in \mathbb{R}_+$, a continuous dynamical system on the phase space $\mathcal H$. More precisely, the associated evolution operator
	$S^\alpha(t):\mathcal H\to\mathcal H$ is given by
	\begin{align*}
		S^\alpha(t)(u_0,u_1)=(u(t),u_t(t)),
		\qquad t\ge0,
	\end{align*}
	where $U=(u,u_t)$ denotes the weak solution of problem \eqref{P} corresponding to the initial data $U_0=(u_0,u_1)$. Throughout the remainder of the paper, we study the long-time dynamics of the family of dynamical systems
	$(\mathcal H,S^\alpha(t))$.
	\subsection{Preliminaries on Global Attractor Theory}
	Let $(X,S(t))$ be a dynamical system. We recall the following standard
	definitions (see, e.g., \cite{Chueshov,chueshov,chueshov-yellow}).
	
	A closed set $B\subset X$ is said to be absorbing for $(X,S(t))$
	if, for every bounded set $D\subset X$, there exists $t_0=t_0(D)\ge0$
	such that
	$$
	S(t)D\subset B,\qquad t\ge t_0(D).
	$$
	The dynamical system $(X,S(t))$ is said to be dissipative if it
	possesses a bounded absorbing set $B$. If $X$ is a Banach space, a
	number $R>0$ is called a radius of dissipativity if
	$$
	B\subset\{x\in X:\|x\|_X\le R\}.
	$$
	
	The dynamical system $(X,S(t))$ is said to be asymptotically
	smooth if, for any bounded set $D\subset X$ such that
	$S(t)D\subset D$ for all $t>0$, there exists a compact set
	$K\subset\overline{D}$ such that
	$$
	\lim_{t\to+\infty}d_X\{S(t)D\mid K\}=0,
	$$
	where
	$$
	d_X\{A\mid B\}:=
	\sup_{x\in A}\operatorname{dist}_X(x,B)
	$$
	denotes the Hausdorff semidistance.
	
	A \textbf{global attractor} for a dynamical system $(X,S(t))$ on a
	complete metric space $X$ is a closed and bounded set $\mathcal{A}\subset X$
	that is invariant, i.e.,
	$$
	S(t)\mathcal{A}=\mathcal{A}
	\quad\text{for every }t>0,
	$$
	and uniformly attracting, i.e.,
	$$
	\lim_{t\to+\infty}\sup_{y\in B}
	\operatorname{dist}_X(S(t)y,\mathcal{A})=0
	\quad\text{for every bounded set }B\subset X.
	$$
	
	The dissipativity and asymptotic smoothness properties provide the
	standard framework for establishing the existence of a compact global
	attractor. In particular, we shall use the following criterion.
	
	\begin{theorem}\cite[Theorem 2.3]{Chueshov}\label{existence}
		Let $(X,S(t))$ be a dissipative dynamical system on a complete metric
		space $X$. Then $(X,S(t))$ possesses a compact global attractor
		$\mathcal{A}$ if and only if $(X,S(t))$ is asymptotically smooth.
	\end{theorem}
	
	Let $\mathcal{N}$ denote the set of stationary points of the dynamical
	system $(X,S(t))$, that is,
	$$
	\mathcal{N}
	=
	\{v\in X:S(t)v=v\quad\text{for all }t\ge0\}.
	$$
	The unstable manifold emanating from $\mathcal{N}$ is defined by
	$$
	\mathcal{M}^u(\mathcal{N})
	=
	\left\{
	y\in X:
	\begin{array}{l}
		\text{there exists a full trajectory }\{u(t):t\in\mathbb{R}\}
		\text{ such that}\\
		u(0)=y\text{ and }
		\operatorname{dist}_X(u(t),\mathcal{N})\to0
		\text{ as }t\to-\infty
	\end{array}
	\right\}.
	$$
	
	A dynamical system $(X,S(t))$ is said to be \textbf{gradient} if it
	possesses a strict Lyapunov function, i.e., there exists a continuous
	functional $\Phi:X\to\mathbb{R}$ such that
	\begin{itemize}
		\item the function $t\mapsto\Phi(S(t)y)$ is nonincreasing for every
		$y\in X$; and
		\item if $\Phi(S(t)y)=\Phi(y)$ for all $t>0$ and some $y\in X$, then
		$S(t)y=y$ for all $t>0$, i.e., $y$ is a stationary point of the
		dynamical system $(X,S(t))$.
	\end{itemize}

	When the dynamical system is gradient, the global attractor admits a
	more precise characterization in terms of the stationary points. In
	particular, it is generated by the unstable manifold of the set of
	stationary points.
	
	\begin{theorem}[{\cite[Theorem 7.5.6]{chueshov-yellow}}]\label{theo-carct}
		Let $(X,S(t))$ be a dynamical system possessing a compact global
		attractor $\mathfrak{A}$. Assume that $(X,S(t))$ admits a strict
		Lyapunov function on $X$. Then
		$$
		\mathfrak{A}=\mathcal{M}^u(\mathcal{N}).
		$$
		Moreover, the global attractor $\mathfrak{A}$ consists of all full
		trajectories $\Upsilon=\{u(t):t\in\mathbb{R}\}$ satisfying
		$$
		\lim_{t\to-\infty}\operatorname{dist}_X(u(t),\mathcal{N})=0
		\quad\text{and}\quad
		\lim_{t\to+\infty}\operatorname{dist}_X(u(t),\mathcal{N})=0.
		$$
	\end{theorem}
	
	\subsection{Statements of Main Result}
	We are now in a position to state the main result of this paper.
	\begin{theorem}{\bf[Global attractors]}\label{theo-main}
		Under Assumption
		\ref{Assumption}, for every
		$\alpha\in\mathbb{R}_+$, the dynamical system
		$(\mathcal{H},S^{\alpha}(t))$ possesses a compact global attractor
		$\mathcal{A}_{\alpha}$ given by
		$$
		\mathcal{A}_{\alpha}
		=
		\mathcal{M}^{u}_{\alpha}(\mathcal{N}),
		$$
		where $\mathcal{N}$ denotes the set of stationary points and
		$\mathcal{M}^{u}_{\alpha}(\mathcal{N})$ is the unstable manifold
		emanating from $\mathcal{N}$.
	\end{theorem}
	\begin{proof}
		The existence of a compact global attractor follows directly from
		Theorem~\ref{existence}. Indeed, Proposition~\ref{Prop-absorbing-set}
		shows that, for every $\alpha\in\mathbb{R}_+$, the dynamical system
		$(\mathcal{H},S(t))$ is dissipative, while Corollary~\ref{Corollary-stab-est}
		implies that it is asymptotically smooth. Hence, Theorem~\ref{existence}
		yields the existence of a compact global attractor $\mathcal{A}_\alpha$. Moreover, Proposition~\ref{gds} shows that $(\mathcal{H},S(t))$ is a
		gradient dynamical system. Therefore, by Theorem~\ref{theo-carct}, the
		global attractor admits the characterization
		$$
		\mathcal{A}_\alpha=\mathcal{M}^u_\alpha(\mathcal{N}),
		.$$
	\end{proof} 
	
	In the remainder of this section, we present the proofs of
	Proposition~\ref{Prop-absorbing-set}, Corollary~\ref{Corollary-stab-est},
	and Proposition~\ref{gds}.
	
	\subsection{Dissipativity}
	Before stating Proposition \ref{Prop-absorbing-set}, we introduce the modified energy functional $\mathcal{L}(t)$ defined by
	\begin{align}\label{4.1}
		\mathcal{L}(t)=E_U(t)+L,
		\qquad
		L:=C_{\nu}|\Omega|
		+\frac{1}{\omega\lambda_1}\|h\|^2 .
	\end{align}
	Proceeding as in the first energy estimate for the Galerkin approximations, it is readily seen that
	\begin{align}\label{4.3}
		\mathcal{L}(t)
		\geq
		\frac{\omega}{4}\|U(t)\|_{\mathcal H}^{2}.
	\end{align}
	\begin{proposition}\label{Prop-absorbing-set}{\bf[Dissipative property]}
		Let $B\subset\mathcal H$ be a bounded set. For every $\alpha\in \mathbb{R}_+$, let
		$U(t)=S^\alpha(t)U_0$ be the weak solution of problem \eqref{P} with
		$U_0\in B$. Under Assumption
		\ref{Assumption}, there exist
		positive constants $K_B$ and $R$ [independent of $\alpha$ and $U_0$], such that
		\begin{align}\label{Unifor-inequ-H}
			||S^\alpha(t)U_0||_{\mathcal H}^{2}
			\leq
			\begin{cases}
				\displaystyle
				\frac{4}{\omega}
				\left[
				\frac{\alpha}{2K_B}(t-1)^+
				+\mathcal{L}(0)^{-\frac{\alpha}{2}}
				\right]^{-\frac{2}{\alpha}}
				+R^2,
				& \text{if }\alpha>0,\\[3ex]
				\displaystyle
				\frac{4}{\omega}
				\sup_{0\le s\le1}\mathcal{L}(s)
				\exp\left(
				-[t]\ln\left(1+\frac{1}{K_B}\right)
				\right)
				+R^2,
				& \text{if }\alpha=0,
			\end{cases}
		\end{align}
		for all $t\ge0$, where $(t-1)^+=\max\{t-1,0\}$.
		This shows that the dynamical system $(\mathcal H,S^\alpha(t))$ possesses a bounded absorbing set. Consequently, $(\mathcal H,S^\alpha(t))$ is dissipative.
	\end{proposition}
	
	\begin{proof}
		Multiplying the Eq. $\eqref{P}$ by $u_t$, and integrating over $\Omega\times [t,t+1]$, we have
		\begin{eqnarray}\label{4.4}
			\int_{t}^{t+1}\|\nabla u_t(s)\|^{\alpha+2}ds= \mathcal{L}(t)-\mathcal{L}(t+1)=: D(t).
		\end{eqnarray}
		Using the Poincar\'e inequality,
		H\"older's inequality with
		$\frac{\alpha}{\alpha+2}+\frac{2}{\alpha+2}=1$,
		and \eqref{4.4}, we have
		\begin{eqnarray}
			\label{4.5}
			\begin{aligned}
				\int_t^{t+1}\|u_t(s)\|^2\,ds
				&\leq
				\frac{1}{\lambda_1}
				\int_t^{t+1}\|\nabla u_t(s)\|^2\,ds \\
				&\leq
				\frac{1}{\lambda_1}
				\left(\int_t^{t+1}\|\nabla u_t(s)\|^{\alpha+2}\,ds\right)^{\frac{2}{\alpha+2}}
				\leq
				\frac{1}{\lambda_1}D(t)^{\frac{2}{\alpha+2}}.
			\end{aligned}
		\end{eqnarray}
		which implies that there exist $t_1\in [t,t+1/4]$ and
		$t_2\in [t+3/4,t+1]$ such that
		\begin{align}\label{4.6}
			\|u_t(t_i)\|^2
			\leq
			4\int_t^{t+1}\|u_t(s)\|^2\,ds
			\leq
			\frac{4}{\lambda_1}D(t)^{\frac{2}{\alpha+2}},
			\qquad \text{for } i=1,2.
		\end{align}
		Next, multiplying Eq. \eqref{P} by $u$ and integrating over
		$\Omega\times[t_1,t_2]$, we obtain
		\begin{align}\label{4.7}
			\begin{aligned}
				\int_{t_1}^{t_2}E_u(s)\,ds
				+\frac12\int_{t_1}^{t_2}\|\nabla u(s)\|^2\,ds
				={}&
				-\int_{t_1}^{t_2}
				\left[
				\int_{\Omega}F(u)\,dx
				-\int_{\Omega}f(u)u\,dx
				\right]ds \\
				&+\frac32\int_{t_1}^{t_2}\|u_t(s)\|^2\,ds
				-\left[\int_{\Omega}u_tu\,dx\right]_{t_1}^{t_2} \\
				&+\int_{t_1}^{t_2}
				\|\nabla u_t(s)\|^\alpha(\Delta u_t,u)\,ds
				+\int_{t_1}^{t_2}\int_{\Omega}hu\,dx\,ds.
			\end{aligned}
		\end{align}
		By \eqref{2.4}, we have
		\begin{align*}
			\int_{t_1}^{t_2}\left[\int_{\Omega}F(u)dx-\int_{\Omega}f(u)udx\right]ds\le \frac{C_\nu}{2\lambda_1}\int_{t_1}^{t_2}\|\nabla u(s)\|^2ds.
		\end{align*}
		Then, combining \eqref{4.1} and \eqref{4.7}, we obtain
		\begin{align}\label{4.8}
			\begin{aligned}
				\int_{t_1}^{t_2}\mathcal{L}(s)\,ds
				+\frac{\omega}{2}\int_{t_1}^{t_2}\|\nabla u(s)\|^2\,ds
				\leq{}&
				\frac{3}{2}\int_{t_1}^{t_2}\|u_t(s)\|^2\,ds
				-\left[\int_{\Omega}u_tu\,dx\right]_{t_1}^{t_2} \\
				&+\int_{t_1}^{t_2}
				\|\nabla u_t(s)\|^\alpha(\Delta u_t,u)\,ds
				+L.
			\end{aligned}
		\end{align}
		It follows directly from \eqref{4.5} that
		\begin{align*}
			\frac{3}{2}\int_{t_1}^{t_2}\|u_t(s)\|^2\,ds
			\leq
			\frac{3}{2\lambda_1}D(t)^{\frac{2}{\alpha+2}}.
		\end{align*}
		Combining Poincar\'e's inequality with estimates \eqref{4.6} and \eqref{4.3}, we obtain
		\begin{eqnarray*}
			-\left[\int_{\Omega}u_tu\,dx\right]_{t_1}^{t_2}
			&\leq&
			\|u_t(t_1)\|\|u(t_1)\|
			+\|u_t(t_2)\|\|u(t_2)\|\\
			&\leq&
			\lambda_1^{-1/2}\|u_t(t_1)\|\|\nabla u(t_1)\|
			+\lambda_1^{-1/2}\|u_t(t_2)\|\|\nabla u(t_2)\|\\
			&\leq&
			4\lambda_1^{-3/2}D(t)^{\frac{1}{\alpha+2}}
			\sup_{t\leq s\leq t+1}\|\nabla u(s)\|\\
			&\leq&
			\frac{8\lambda_1^{-3/2}}{\omega^{1/2}}
			D(t)^{\frac{1}{\alpha+2}}
			\sup_{t\leq s\leq t+1}\mathcal{L}(s)^{1/2}\\
			&\leq&
			\frac{128\lambda_1^{-3}}{\omega}
			D(t)^{\frac{2}{\alpha+2}}
			+\frac{1}{8}
			\sup_{t\leq s\leq t+1}\mathcal{L}(s).
		\end{eqnarray*}
		From H\"older inequality with  $\frac{\alpha+1}{\alpha+2}+\frac{1}{\alpha+2}=1$, Young's inequality, estimates \eqref{4.3}-\eqref{4.4}, we have
		\begin{eqnarray*}
			\int_{t_1}^{t_2}\|\nabla u_t(s)\|^{\alpha}(\Delta u_t,u)ds&=& -\int_{t_1}^{t_2}\|\nabla u_t(s)\|^{\alpha}(\nabla u_t,\nabla u)ds\\
			&\le & \int_{t_1}^{t_2}\|\nabla u_t(s)\|^{\alpha+1}\|\nabla u(s)\|ds\\
			&\le&\left(\int_{t_1}^{t_2}\|\nabla u_t(s)\|^{\alpha+2}ds\right)^{\frac{\alpha+1}{\alpha+2}}\left(\int_{t_1}^{t_2}\|\nabla u(s)\|^{\alpha+2}ds\right)^{\frac{1}{\alpha+2}}\\
			&\le& D(t)^{\frac{\alpha+1}{\alpha+2}}\sup_{0\le s\le t+1}\|\nabla u(s)\|\\
			&\le&\frac{2 }{\omega^{1/2}}D(t)^{\frac{\alpha+1}{\alpha+2}}\sup_{0\le s\le t+1}\mathcal{L}(s)^{1/2}\\
			&\le&\frac{8 }{\omega}D(t)^{\frac{2(\alpha+1)}{\alpha+2}}+\frac{1}{8}\sup_{0\le s\le t+1}\mathcal{L}(s)
		\end{eqnarray*}
		Inserting  the last four estimates into \eqref{4.8} yields
		\begin{eqnarray}\label{4.9}
			\begin{aligned}
				\int_{t_1}^{t_2}\mathcal{L}(s)ds\le C_0\left[\,D(t)^{\frac{2}{\alpha+2}}+D(t)^{\frac{2(\alpha+1)}{\alpha+2}}\right]+\frac{1}{4}\sup_{0\le s\le t+1}\mathcal{L}(s)+L,
			\end{aligned}
		\end{eqnarray}
		where $C_0=\frac{3}{2\lambda_1}+\frac{128\lambda_1^{-3}}{\omega}+\frac{8}{\omega}$.
		
		Due to $\mathcal{L}(s)$  is decreasing, by the Mean Value Theorem, there exists $\tau\in [t_1,t_2]$ such that
		\begin{align*}
			\int_{t_1}^{t_2}\mathcal{L}(s)ds=\mathcal{L}(\tau)(t_2-t_1)\ge \frac{1}{2}\mathcal{L}(t+1).
		\end{align*}
		Thus, from \eqref{4.4}, we have
		\begin{eqnarray}\label{4.10}
			\mathcal{L}(t)=\mathcal{L}(t+1)+D(t)\le2\int_{t_1}^{t_2}\mathcal{L}(s)ds+D(t).
		\end{eqnarray}
		Substituting \eqref{4.10} in \eqref{4.9}, we obtain that
		\begin{align}\begin{split}
				\label{4.11}
				\mathcal{L}(t)\le \left[2C_0D(t)^{\frac{2}{\alpha+2}}+2C_0D(t)^{\frac{2(\alpha+1)}{\alpha+2}}+D(t)\right]+\frac{1}{2}\sup_{0\le s\le t+1}\mathcal{L}(s)+2L.
		\end{split}\end{align}
		Using that $\mathcal{L}(t)=\sup_{t\le s\le t+1}\mathcal{L}(s)$, it follows from \eqref{4.11} that
		\begin{eqnarray}
			\label{4.12}
			\mathcal{L}(t)\le D(t)^{\frac{2}{\alpha+2}}\left[\,4C_0+4C_0D(t)^{\frac{2\alpha}{\alpha+2}}+2D(t)^{\frac{\alpha}{\alpha+2}}\right]+4L.
		\end{eqnarray}
		We infer form  the definition of $D(t)$ that
		$$\sup_{U_0\in B}\left[4C_0+4C_0D(t)^{\frac{2\alpha}{\alpha+2}}+2D(t)^{\frac{\alpha}{\alpha+2}}\right]\le C_B.$$
		Hence, \eqref{4.12} can be rewritten as
		\begin{eqnarray*}
			\mathcal{L}(t)\le C_BD(t)^{\frac{2}{\alpha+2}}+4L,
		\end{eqnarray*}
		which implies that
		\begin{align}
			\label{4.13}
			\mathcal{L}(t)^{1+\frac{\alpha}{2}}
			\le
			K_B\bigl[\mathcal{L}(t)-\mathcal{L}(t+1)\bigr]
			+(8L)^{1+\frac{\alpha}{2}},
		\end{align}
		where
		$
		K_B:=2^{\frac{\alpha}{2}}(C_B)^{\frac{\alpha+2}{2}}.
		$
		
		Applying Nakao's lemma \cite[Lemma 2.1]{Nakao} to \eqref{4.13}, we obtain
		\begin{align}\label{4.14}
			\mathcal{L}(t)\le 8L+\Phi_\alpha(t),
			\qquad 0\le t<T,
		\end{align}
		where
		$$
		\Phi_\alpha(t)=
		\begin{cases}
			\displaystyle
			\left[
			\frac{\alpha}{2K_B}(t-1)^+
			+\mathcal{L}(0)^{-\frac{\alpha}{2}}
			\right]^{-\frac{2}{\alpha}},
			&\text{if } \alpha>0,\\[2ex]
			\displaystyle
			\sup_{0\le s\le1}\mathcal{L}(s)
			\exp\left(
			-[t]\ln\left(1+\frac{1}{K_B}\right)
			\right),
			& \text{if }\alpha=0,
		\end{cases}
		$$
		and $(t-1)^+=\max\{t-1,0\}$. This proves \eqref{Unifor-inequ-H}. Moreover, choosing
		$
		R^2:=\frac{32L}{\omega},
		$
		the closed ball
		\begin{align*}
			\mathcal{B}
			=\{U\in\mathcal H:\|U\|_{\mathcal H}\le R\}
		\end{align*}
		is a bounded absorbing set for the dynamical system
		$(\mathcal H,S^\alpha(t))$. This completes the proof of Proposition
		\ref{Prop-absorbing-set}.
	\end{proof}
	\begin{remark}
		Multiplying the Eq. $\eqref{P}$ by $u_t$, and integrating over $\Omega\times [0,t]$, we have
		\begin{eqnarray}\label{4.16}
			\underbrace{\mathcal{L}(t)}_{\ge 0}+\int_{0}^{t}\|\nabla u_t(s)\|^{\alpha+2}ds= \mathcal{L}(0).
		\end{eqnarray}
		Then, combining \eqref{4.16} and \eqref{4.14} there exists $t_B>0$ such that
		\begin{align*}
			\int_0^t\|\nabla u_t(s)\|^{\alpha+2}ds\le 8L, \quad \forall t\ge t_B.
		\end{align*}
	\end{remark}

	\subsection{Smoothness Property}
	To establish the asymptotic smoothness of the dynamical system
	$(\mathcal{H},S^{\alpha}(t))$ established in Corollary~\ref{Corollary-stab-est},
	we shall employ the following compactness criterion, which is a direct consequence of the stabilizability estimate obtained in Proposition~\ref{Prop-stab-est-0}. This criterion is due to
	\cite[Proposition 2.10]{Chueshov}.
	
	\begin{theorem}\cite[Proposition 2.10]{Chueshov}
		Let $(X,S(t))$ be a dynamical system on a complete metric space $X$
		endowed with a metric $d$. Assume that, for any bounded positively
		invariant set $B\subset X$ and any $\varepsilon>0$, there exists
		$T=T(\varepsilon,B)>0$ such that
		\begin{align*}
			d(S(T)y_1,S(T)y_2)
			\leq \varepsilon+\Psi_{\varepsilon,B,T}(y_1,y_2),
			\qquad y_i\in B,
		\end{align*}
		where $\Psi_{\varepsilon,B,T}$ is a function defined on $B\times B$
		such that
		\begin{align*}
			\lim_{n\to+\infty}\lim_{m\to+\infty}
			\Psi_{\varepsilon,B,T}(y_n,y_m)=0
		\end{align*}
		for every sequence $\{y_n\}$ in $B$. Then $(X,S(t))$ is asymptotically
		smooth.
	\end{theorem}
	
	The following stabilizability estimate provides the key ingredient for
	applying this criterion to the dynamical system generated by
	\eqref{P}.
	\begin{proposition}{\bf[Stabilizability estimate]}
		\label{Prop-stab-est-0}
		Let $B\subset\mathcal{H}$ be a bounded set. Under Assumption
		\ref{Assumption}, for every $\alpha\in\mathbb{R}_+$, let
		$U^1(t)=S^\alpha(t)U^1(0)$ and $U^2(t)=S^\alpha(t)U^2(0)$ be weak solutions of
		problem \eqref{P} corresponding to initial data $U^1(0)=U^1_0,U^2(0)=U^2_0\in B$.
		Then there exist positive constants $\varepsilon$ and $C_B$ such that
		\begin{equation}\label{entropy-alpha}
			\begin{aligned}
				\|S^{\alpha}(t)U^1_0-S^{\alpha}(t)U^2_0\|_{\mathcal H}^{2}
				\leq{}&
				C_B e^{-\varepsilon t}
				\|U^1(0)-U^2(0)\|_{\mathcal H}^{2}
				\\
				&+
				C_B\int_0^t
				e^{-\varepsilon(t-s)}
				\|S^{\alpha}(s)U^1_0-S^{\alpha}(s)U^2_0\|_{\mathcal H_{-1}}^{\frac{\alpha+2}{\alpha+1}}
				\,ds,
			\end{aligned}
		\end{equation}
		where $\mathcal{H}_{-1}=L^2(\Omega)\times H^{-1}(\Omega)$.
	\end{proposition}
	
	\begin{proof}
		Let us denote $z:=u^1-u^2$. Then $z(t)=(z(t),z_t(t))=S^{\alpha}(t)U^1_0-S^{\alpha}(t)U^2_0$ satisfies the following equation
		
		\begin{align}\label{5.5}
			\begin{cases}
				z_{tt}-\Delta z
				+\left[
				\|\nabla u^1_t\|^\alpha\Delta u^1_t
				-\|\nabla u^2_t\|^\alpha\Delta u^2_t
				\right]
				+f(u^1)-f(u^2)=0,
				& \text{in }\Omega\times\mathbb{R}_+,\\
				z=0,
				& \text{on }\Gamma\times\mathbb{R}_+,\\
				z(x,0)=u^1_0-u^2_0,\qquad
				z_t(x,0)=u^1_1-u^2_1,
				& x\in\Omega.
			\end{cases}
		\end{align}
		First, observe that we can rewrite
		\begin{eqnarray*}
			\|\nabla u^1_t(t)\|^\alpha\Delta u^1_t
			-\|\nabla u^2_t(t)\|^\alpha\Delta u^2_t
			&=&
			\frac12
			\left[\,
			\|\nabla u^1_t(t)\|^\alpha
			+
			\|\nabla u^2_t(t)\|^\alpha
			\,\right]\Delta u^1_t\\
			&&+
			\frac12
			\left[\,
			\|\nabla u^1_t(t)\|^\alpha
			-
			\|\nabla u^2_t(t)\|^\alpha
			\,\right]
			\left[
			\Delta u^1_t+\Delta u^2_t
			\right].
		\end{eqnarray*}
		Then, multiplying \eqref{5.5} by $z_t$ and integrating over
		$\Omega$, we obtain
		\begin{equation}
			\label{asymptotic00-alpha}
			\begin{aligned}
				&
				\frac12\frac{d}{dt}
				\left[
				\|Z(t)\|_{\mathcal H}^2
				+
				\int_\Omega\int_0^1
				f'(\theta u^1+(1-\theta)u^2)
				\,d\theta\,|z|^2dx
				\right]
				\\
				&\quad+
				\underbrace{\frac12
					\left[\,
					\|\nabla u^1_t(t)\|^\alpha
					+
					\|\nabla u^2_t(t)\|^\alpha
					\,\right]
					\|\nabla z_t(t)\|^2}_{I}
				\\
				&\quad
				+
				\underbrace{\frac12
					\left[\,
					\|\nabla u^1_t(t)\|^\alpha
					-
					\|\nabla u^2_t(t)\|^\alpha
					\,\right]
					\left[\,
					\|\nabla u^1_t(t)\|^2-
					\|\nabla u^2_t(t)\|^2\,\right]}_{\ge 0}
				\\
				&=
				\underbrace{\frac12
					\int_\Omega
					\int_0^1
					f''(\theta u^1+(1-\theta)u^2)
					(\theta u^1_t+(1-\theta)u^2_t)
					|z|^2
					\,d\theta\,dx}_{J}.
			\end{aligned}
		\end{equation}
		Since the function $s\mapsto s^\alpha$ is increasing, the third term on the left-hand side is nonnegative. Moreover, using
		\begin{align*}
			\begin{cases}
				(a+b)^\alpha\le a^\alpha+b^\alpha, & 0<\alpha<1,\\
				(a+b)^\alpha\le 2^{\alpha-1}(a^\alpha+b^\alpha), & \alpha\ge1,
			\end{cases}
		\end{align*}
		we obtain
		\begin{eqnarray*}
			I
			&=&
			\frac12
			\left[\,
			\|\nabla u^1_t(t)\|^\alpha
			+
			\|\nabla u^2_t(t)\|^\alpha
			\,\right]
			\|\nabla z_t(t)\|^2
			\\
			&\ge&
			Q_\alpha\|\nabla z_t(t)\|^{\alpha+2}
			+
			\frac14
			\left[\,
			\|\nabla u^1_t(t)\|^\alpha
			+
			\|\nabla u^2_t(t)\|^\alpha
			\,\right]
			\|\nabla z_t(t)\|^2,
		\end{eqnarray*}
		where
		\begin{align*}
			Q_\alpha=
			\begin{cases}
				\dfrac14, & 0<\alpha<1,\\[2mm]
				\dfrac{1}{2^{\alpha+1}}, & \alpha\ge1.
			\end{cases}
		\end{align*}
		On the other hand, by the assumption \eqref{hyp_f'}, the inequality H\"{o}lder 
		(with $\frac12+\frac16+\frac13=1$),
		the Sobolev embedding
		$H_0^1(\Omega)\hookrightarrow L^6(\Omega)$,
		and Young's inequality, we obtain
		\begin{eqnarray*}
			J
			&\le&
			C
			\int_\Omega
			\left(1+|u^1|^3+|u^2|^3\right)
			\left(|u^1_t|+|u^2_t|\right)
			|z|^2dx
			\\
			&\le&
			C
			\left[
			1+\|u^1(t)\|_6^3+\|u^2(t)\|_6^3
			\right]
			\left[
			\|u^1_t(t)\|_6+\|u^2_t(t)\|_6
			\right]
			\|z(t)\|_6^2
			\\
			&\le&
			C
			\left[
			1+\|\nabla u^1(t)\|^3+\|\nabla u^2(t)\|^3
			\right]
			\left[
			\|\nabla u^1_t(t)\|+\|\nabla u^2_t(t)\|
			\right]
			\|\nabla z(t)\|^2
			\\
			&\le&
			C_B
			\left[
			\|\nabla u^1_t(t)\|
			+\|\nabla u^2_t(t)\|
			\right]
			\|\nabla z(t)\|^2.
		\end{eqnarray*}
		Consequently, using the previous estimates for $I$ and $J$, it follows from \eqref{asymptotic00-alpha} that
		\begin{equation}
			\label{asymptotic01-alpha}
			\begin{aligned}
				&
				\frac12
				\frac{d}{dt}
				\left[
				\|Z(t)\|_{\mathcal H}^2
				+
				\int_\Omega
				\int_0^1
				f'(\theta u^1+(1-\theta)u^2)
				\,d\theta\,|z|^2dx
				\right]
				\\
				&\quad
				+Q_\alpha\|\nabla z_t(t)\|^{\alpha+2}
				+\frac14
				\left[
				\|\nabla u^1_t(t)\|^\alpha
				+
				\|\nabla u^2_t(t)\|^\alpha
				\right]
				\|\nabla z_t(t)\|^2
				\\
				&\le
				C_B
				\left[
				\|\nabla u^1_t(t)\|
				+
				\|\nabla u^2_t(t)\|
				\right]
				\|\nabla z(t)\|^2.
			\end{aligned}
		\end{equation}
		Next, multiplying \eqref{5.5} by $\varepsilon z$, integrating over
		$\Omega$, and integrating by parts, we obtain
		
		\begin{align}
			\label{asymptotic1}
			\varepsilon\|Z(t)\|_{\mathcal H}^2
			+
			\varepsilon
			\int_\Omega
			\int_0^1
			f'(\theta u^1+(1-\theta)u^2)
			\,d\theta\,|z|^2dx
			+
			\varepsilon
			\frac{d}{dt}
			\int_\Omega
			z_tz\,dx
			=
			\sum_{j=1}^3L_j,
		\end{align}
		
		where
		
		\begin{align*}
			L_1&=
			2\varepsilon\|z_t(t)\|^2,
			\\
			L_2&=
			-\frac{\varepsilon}{2}
			\left[\,
			\|\nabla u^1_t(t)\|^{\alpha}
			+
			\|\nabla u^2_t(t)\|^{\alpha}
			\,\right]
			\int_\Omega
			\nabla z_t\cdot\nabla z\,dx,
			\\
			L_3&=
			-\frac{\varepsilon}{2}
			\left[\,
			\|\nabla u^1_t(t)\|^{\alpha}
			-
			\|\nabla u^2_t(t)\|^{\alpha}
			\,\right]
			\int_\Omega
			(\nabla u^1_t+\nabla u^2_t)\cdot\nabla z\,dx.
		\end{align*}
		Adding \eqref{asymptotic01-alpha} and \eqref{asymptotic1}, we obtain
		
		\begin{equation}
			\label{asymptotic02-alpha}
			\begin{aligned}
				&
				\frac{d}{dt}
				\left[
				\frac12
				\|Z(t)\|_{\mathcal H}^2
				+
				\frac12
				\int_\Omega
				\int_0^1
				f'(\theta u^1+(1-\theta)u^2)
				\,d\theta\,|z|^2dx
				+
				\varepsilon
				\int_\Omega
				z_tz\,dx
				\right]
				\\
				&\quad
				+
				\varepsilon
				\|Z(t)\|_{\mathcal H}^2
				+
				\varepsilon
				\int_\Omega
				\int_0^1
				f'(\theta u^1+(1-\theta)u^2)
				\,d\theta\,|z|^2dx
				\\
				&\quad
				+Q_\alpha
				\|\nabla z_t(t)\|^{\alpha+2}
				+
				\frac14
				\left[
				\|\nabla u^1_t(t)\|^\alpha
				+
				\|\nabla u^2_t(t)\|^\alpha
				\right]
				\|\nabla z_t(t)\|^2
				\\
				&\le
				C_{\mathcal B}
				\left[
				\|\nabla u^1_t(t)\|
				+
				\|\nabla u^2_t(t)\|
				\right]
				\|\nabla z(t)\|^2
				+
				\sum_{j=1}^3L_j.
			\end{aligned}
		\end{equation}
		
		Now, we define the perturbed energy functional
		
		\begin{align}\label{energy-Ew}
			E_Z^\varepsilon(t)
			=
			\frac12
			\|Z(t)\|_{\mathcal H}^2
			+
			\frac12
			\int_\Omega
			\int_0^1
			f'(\theta u^1+(1-\theta)u^2)
			\,d\theta\,|z|^2dx
			+
			\varepsilon
			\int_\Omega
			z_tz\,dx
			+
			K_f\|z(t)\|^2,
		\end{align}
		where $K_f>0$ is defined as in \eqref{lower1} of Remark \ref{lowerbounded}. Choosing $\varepsilon>0$ sufficiently small and using Young's inequality together with the dissipativity assumption \eqref{hyp-inf-f}, there exists a constant $\varrho>0$ such that
		
		\begin{align}\label{eq-energy_Ew}
			\varrho
			\|Z(t)\|_{\mathcal H}^2
			\leq
			E_Z^\varepsilon(t)
			\leq
			C_{\mathcal B}
			\|Z(t)\|_{\mathcal H}^2.
		\end{align}
		Combining the definition of $E_Z^\varepsilon$ in \eqref{energy-Ew} with estimate \eqref{asymptotic02-alpha}, we obtain
		\begin{equation}
			\label{asymptotic03}
			\begin{aligned}
				&
				\frac{d}{dt}E_Z^\varepsilon(t)
				+
				2\varepsilon E_Z^\varepsilon(t)
				+
				Q_{\alpha}
				\|\nabla z_t(t)\|^{\alpha+2}
				+
				\frac14
				\left[\,
				\|\nabla u^1_t(t)\|^\alpha
				+
				\|\nabla u^2_t(t)\|^\alpha
				\,\right]
				\|\nabla z_t(t)\|^2
				\\
				&\le
				\underbrace{2\varepsilon^2
					\left|
					\int_\Omega z_tz\,dx
					\right|}_{L}
				+
				2\varepsilon K_f\|z\|^2
				+
				C_{\mathcal B}
				\left[\,
				\|\nabla u^1_t(t)\|
				+
				\|\nabla u^2_t(t)\|
				\,\right]
				\|\nabla z(t)\|^2
				+
				\sum_{j=1}^3L_j.
			\end{aligned}
		\end{equation}
		
		First, by Poincar\'{e}'s inequality and Young's inequality with conjugate exponents
		$p=\alpha+2$ and $q=\frac{\alpha+2}{\alpha+1}$, we obtain
		\begin{align*}
			L
			\le
			\frac{Q_\alpha}{2}
			\|\nabla z_t(t)\|^{\alpha+2}
			+
			C_\alpha
			\varepsilon^{\frac{2(\alpha+2)}{\alpha+1}}
			\|z(t)\|^{\frac{\alpha+2}{\alpha+1}},
		\end{align*}
		Using the interpolation inequality
		$
		H_0^1(\Omega)\hookrightarrow L^2(\Omega)\hookrightarrow H^{-1}(\Omega),
		$
		together with Young's inequality, we obtain
		\begin{align*}
			L_1
			\le
			C\varepsilon
			\|\nabla z_t(t)\|
			\|z_t(t)\|_{H^{-1}}
			\le
			\frac{Q_\alpha}{2}
			\|\nabla z_t(t)\|^{\alpha+2}
			+
			C'_{\alpha}\varepsilon^{\frac{\alpha+2}{\alpha+1}}
			\|z_t(t)\|_{H^{-1}}^{\frac{\alpha+2}{\alpha+1}}.
		\end{align*}
		Next, by Young's inequality,
		\begin{eqnarray*}
			L_2
			&\le&
			\frac{\varepsilon}{2}
			\left[\,
			\|\nabla u^1_t(t)\|^\alpha
			+
			\|\nabla u^2_t(t)\|^\alpha
			\,\right]
			\|\nabla z_t(t)\|
			\|\nabla z(t)\|
			\\
			&\le&
			\frac18
			\left[\,
			\|\nabla u^1_t(t)\|^\alpha
			+
			\|\nabla u^2_t(t)\|^\alpha
			\,\right]
			\|\nabla z_t(t)\|^2
			+
			\frac{\varepsilon^2}{2}
			\left[\,
			\|\nabla u^1_t(t)\|^\alpha
			+
			\|\nabla u^2_t(t)\|^\alpha
			\,\right]
			\|\nabla z(t)\|^2.
		\end{eqnarray*}
		Finally, for each fixed $t$, if
		$\|\nabla u^1_t(t)\|=\|\nabla u^2_t(t)\|$, then $L_3=0$. Otherwise,
		\begin{align*}
			L_3&=-
			\frac{\varepsilon}{2}
			\left[\,
			\|\nabla u^1_t(t)\|^{\alpha}
			-
			\|\nabla u^2_t(t)\|^{\alpha}
			\,\right]
			\int_\Omega
			(\nabla u^1_t+\nabla u^2_t)\cdot\nabla z\,dx\\
			&\leq
			\frac{\varepsilon}{2}\left|
			\frac{\|\nabla u^1_t(t)\|^{\alpha}
				-
				\|\nabla u^2_t(t)\|^{\alpha}}
			{\|\nabla u^1_t(t)\|
				-
				\|\nabla u^2_t(t)\|}
			\left[
			\|\nabla u^1_t(t)\|
			-
			\|\nabla u^2_t(t)\|
			\right]\right|
			\|\nabla u^1_t+\nabla u^2_t\|
			\|\nabla z\|\\
			&\leq
			\frac{\varepsilon\max\{1,\alpha\}}{2}
			\frac{\|\nabla u^1_t(t)\|^{\alpha}
				+
				\|\nabla u^2_t(t)\|^{\alpha}}
			{\|\nabla u^1_t(t)\|
				+
				\|\nabla u^2_t(t)\|}
			\|\nabla z_t(t)\|
			\left[
			\|\nabla u^1_t\|+\|\nabla u^2_t\|
			\right]\|\nabla z\|\\
			&=
			\frac{\varepsilon\max\{1,\alpha\}}{2}
			\left[
			\|\nabla u^1_t(t)\|^{\alpha}
			+
			\|\nabla u^2_t(t)\|^{\alpha}
			\right]
			\|\nabla z_t(t)\|
			\|\nabla z(t)\|\\
			&\leq
			\frac18
			\left[\,
			\|\nabla u^1_t(t)\|^\alpha
			+
			\|\nabla u^2_t(t)\|^\alpha
			\,\right]
			\|\nabla z_t(t)\|^2
			+
			\frac{(\varepsilon\max\{1,\alpha\})^2}{2}
			\left[\,
			\|\nabla u^1_t(t)\|^\alpha
			+
			\|\nabla u^2_t(t)\|^\alpha
			\,\right]
			\|\nabla z(t)\|^2.
		\end{align*}
		where we have used
		$\left|\|\nabla u^1_t(t)\|-\|\nabla u^2_t(t)\|\right|
		\leq\|\nabla u^1_t(t)-\nabla u^2_t(t)\|
		\leq\|\nabla z_t(t)\|$
		and the inequality (see Lemma 5.5 in \cite{Z-Y})
		\begin{align*}
			0<
			\frac{a^{\alpha}-b^{\alpha}}{a-b}
			\leq
			\max\{1,\alpha\}
			\left(\frac{a^{\alpha}+b^{\alpha}}{a+b}\right),
			\quad
			\forall\alpha>0,\quad a,b\geq0,\quad a+b>0,\quad a\neq b.
		\end{align*}
		
		Collecting the estimates for the terms $L$, $L_1$, $L_2$, and $L_3$ and
		substituting them into \eqref{asymptotic03}, we obtain
		
		\begin{equation}
			\label{asymptotic04}
			\begin{aligned}
				&
				\frac{d}{dt}E_Z^\varepsilon(t)
				+
				2\varepsilon E_Z^\varepsilon(t)
				\leq
				C_{\mathcal B}
				\left[\,
				\|\nabla u^1_t(t)\|
				+
				\|\nabla u^2_t(t)\|
				\,\right]
				\|\nabla z(t)\|^2
				\\
				&
				\qquad
				+
				C
				\left[\,
				\|\nabla u^1_t(t)\|^\alpha
				+
				\|\nabla u^2_t(t)\|^\alpha
				\,\right]
				\|\nabla z(t)\|^2
				+
				C_{\mathcal B}
				\left[\,
				\|z(t)\|^{\frac{\alpha+2}{\alpha+1}}
				+
				\|z_t(t)\|_{H^{-1}}^{\frac{\alpha+2}{\alpha+1}}
				\,\right].
			\end{aligned}
		\end{equation}
		
		Using the equivalence of the perturbed energy \eqref{eq-energy_Ew}, we have
		\begin{align*}
			\|\nabla z(t)\|^2
			\leq
			\frac{1}{\varrho}E_Z^\varepsilon(t).
		\end{align*}
		Consequently, by Young's inequality,
		\begin{eqnarray*}
			&&
			C_{\mathcal B}
			\left[\,
			\|\nabla u^1_t(t)\|
			+
			\|\nabla u^2_t(t)\|
			\,\right]
			\|\nabla z(t)\|^2
			+
			C
			\left[\,
			\|\nabla u^1_t(t)\|^\alpha
			+
			\|\nabla u^2_t(t)\|^\alpha
			\,\right]
			\|\nabla z(t)\|^2
			\\
			&&\leq
			\left[
			\varepsilon
			+
			C_{\mathcal B,\varepsilon}
			\left(
			\|\nabla u^1_t(t)\|^{\alpha+2}
			+
			\|\nabla u^2_t(t)\|^{\alpha+2}
			\right)
			\right]
			E_Z^\varepsilon(t).
		\end{eqnarray*}
		Returning to \eqref{asymptotic04}, we obtain
		
		\begin{equation}
			\label{asymptotic05}
			\begin{aligned}
				\frac{d}{dt}E_Z^\varepsilon(t)
				\leq
				\Theta_{\varepsilon}(t)E_Z^\varepsilon(t)
				+
				C_{\mathcal B}
				\|Z(t)\|_{\mathcal{H}_{-1}}^{\frac{\alpha+2}{\alpha+1}},
			\end{aligned}
		\end{equation}
		where
		\begin{align*}
			\Theta_{\varepsilon}(t)
			=
			-\varepsilon
			+
			C_{\mathcal B,\varepsilon}
			\left[\,
			\|\nabla u^1_t(t)\|^{\alpha+2}
			+
			\|\nabla u^2_t(t)\|^{\alpha+2}
			\,\right],
		\end{align*}
		and
		\begin{align*}
			\|Z(t)\|_{\mathcal{H}_{-1}}
			=
			\sqrt{
				\|z(t)\|^2
				+
				\|z_t(t)\|_{H^{-1}}^2
			}.
		\end{align*}
		
		Since
		\begin{align*}
			\|\nabla u^1_t\|^{\alpha+2},
			\quad
			\|\nabla u^2_t\|^{\alpha+2}
			\in L^1(0,T),
		\end{align*}
		it follows that
		\begin{align*}
			\Theta_{\varepsilon}\in L^1(0,T).
		\end{align*}
		Therefore, applying Gronwall's inequality to \eqref{asymptotic05}, we obtain
		
		\begin{align*}
			E_Z^\varepsilon(t)
			\leq
			E_Z^\varepsilon(0)
			e^{\int_0^t\Theta_{\varepsilon}(s)\,ds}
			+
			C_{\mathcal B}
			\int_0^t
			e^{\int_s^t\Theta_{\varepsilon}(\tau)\,d\tau}
			\|Z(s)\|_{\mathcal{H}_{-1}}^{\frac{\alpha+2}{\alpha+1}}
			\,ds.
		\end{align*}
		
		Since
		\begin{align*}
			\int_s^t\Theta_{\varepsilon}(\tau)\,d\tau
			=
			-\varepsilon(t-s)
			+
			C_{\mathcal B,\varepsilon}
			\int_s^t
			\left[\,
			\|\nabla u^1_t(\tau)\|^{\alpha+2}
			+
			\|\nabla u^2_t(\tau)\|^{\alpha+2}
			\,\right]\,d\tau,
		\end{align*}
		and the trajectories belong to the absorbing set, there exists a constant
		$C_B>0$ such that
		\begin{align*}
			e^{
				C_{\mathcal B,\varepsilon}
				\int_s^t
				\left[\,
				\|\nabla u^1_t(\tau)\|^{\alpha+2}
				+
				\|\nabla u^2_t(\tau)\|^{\alpha+2}
				\,\right]\,d\tau}
			\leq
			C_B,
			\qquad \forall\,0\leq s\leq t.
		\end{align*}
		
		Consequently,
		
		\begin{align}
			\label{asymptotic07}
			E_Z^\varepsilon(t)
			\leq
			C_B e^{-\varepsilon t}
			E_Z^\varepsilon(0)
			+
			C_B
			\int_0^t
			e^{-\varepsilon(t-s)}
			\|Z(s)\|_{\mathcal{H}_{-1}}^{\frac{\alpha+2}{\alpha+1}}
			\,ds.
		\end{align}
		
		Thus, it follows from \eqref{asymptotic07} and \eqref{eq-energy_Ew} that
		
		\begin{align}
			\label{asymptotic08}
			\|Z(t)\|_{\mathcal H}^2
			\leq
			C_B e^{-\varepsilon t}
			\|Z(0)\|_{\mathcal H}^2
			+
			C_B
			\int_0^t
			e^{-\varepsilon(t-s)}
			\|Z(s)\|_{\mathcal H_{-1}}^{\frac{\alpha+2}{\alpha+1}}
			\,ds.
		\end{align}
		Since $Z(t)=S^{\alpha}(t)U^1_0-S^{\alpha}(t)U^2_0$, estimate
		\eqref{entropy-alpha} follows, which completes the proof.
	\end{proof}
	
	\begin{corollary}{\bf[Smoothness property]}\label{Corollary-stab-est}
		Under Assumption
		\ref{Assumption}, the dynamical system $(\mathcal{H},S^{\alpha}(t))$ is asymptotically smooth.
	\end{corollary}
	\begin{proof}
		
		Given $\epsilon>0$, choose $T=T_{\epsilon,B}>0$ sufficiently large such that
		\begin{align*}
			C_B e^{-\varepsilon T}
			\sup_{U_0^1,U_0^2\in B}
			\|U_0^1-U_0^2\|_{\mathcal H}^2
			<\epsilon.
		\end{align*}
		Then, by \eqref{entropy-alpha}, for every $U_0^1,U_0^2\in B$,
		\begin{align*}
			\|S^\alpha(T)U_0^1-S^\alpha(T)U_0^2\|_{\mathcal H}^2
			\leq
			\epsilon
			+
			C_B
			\int_0^T
			e^{-\varepsilon(T-s)}
			\|S^\alpha(s)U_0^1-S^\alpha(s)U_0^2\|_{\mathcal H_{-1}}^{\frac{\alpha+2}{\alpha+1}}
			\,ds.
		\end{align*}
		Define the functional
		$\Psi_T:\mathcal H\times\mathcal H\to\mathbb R_+$ by
		\begin{align*}
			\Psi_T(U_0^1,U_0^2)
			:=
			C_B
			\int_0^T
			e^{-\varepsilon(T-s)}
			\|S^\alpha(s)U_0^1-S^\alpha(s)U_0^2\|_{\mathcal H_{-1}}^{\frac{\alpha+2}{\alpha+1}}
			\,ds.
		\end{align*}
		It follows that
		\begin{align*}
			\|S^\alpha(T)U_0^1-S^\alpha(T)U_0^2\|_{\mathcal H}^2
			\leq
			\epsilon+\Psi_T(U_0^1,U_0^2),
		\end{align*}
		for all $U_0^1,U_0^2\in B$. It remains to show that $\Psi_T$ is a contractive function on $B$. Let
		$\{U_0^n\}_{n\in\mathbb N}\subset B$. Since $B$ is bounded in $\mathcal H$ and
		the embedding $\mathcal H\hookrightarrow\mathcal H_{-1}$ is compact, while
		$\{S^\alpha(\cdot)U_0^n\}_{n\in\mathbb N}$ is uniformly bounded and
		equicontinuous in $\mathcal H_{-1}$ on $[0,T]$, the Arzel\`a--Ascoli theorem
		yields relative compactness in $C([0,T];\mathcal H_{-1})$. Hence, along a
		subsequence,
		\begin{align*}
			\lim_{n\to\infty}\lim_{m\to\infty}
			\Psi_T(U_0^n,U_0^m)=0.
		\end{align*}
		Therefore,
		\begin{align*}
			\liminf_{n\to\infty}\liminf_{m\to\infty}
			\Psi_T(U_0^n,U_0^m)=0,
		\end{align*}
		and $\Psi_T$ is contractive on $B$.
	\end{proof}

	\subsection{Gradient System}
	\begin{proposition}\label{gds}{\bf[Gradient system]}
		Under Assumption
		\ref{Assumption}, $(\mathcal H,S^{\alpha}(t))$ is a gradient dynamical system.
	\end{proposition}
	\begin{proof}
		Energy equality \eqref{ei} shows that $E_U(t)=E(U(t))$ is non increasing. Suppose $E(S(t) U_0) = E(U_0)$ for all $t \geq 0$. Then from energy equality \eqref{ei}, we have
		\begin{equation}\label{gs}
			\int_0^t \|\nabla u_t(s)\|^{\alpha+2} \, ds = 0.
		\end{equation}
		Then, using embedding $H_0^1(\Omega) \hookrightarrow L^2(\Omega)$, it follows from \eqref{gs} that
		$$
		\|u_t(t)\| = 0 \quad \text{for almost everywhere } t > 0.
		$$
		Moreover, it follows from $u_t \in C([0, T]; L^2)$ for all $T > 0$ that $\|u_t(t)\| = 0$ for all $t > 0$, which implies that $U_0 \in \mathcal N$, where $\mathcal N$ is the set of stationary points of the dynamical system $(\mathcal H,S^{\alpha}(t))$. Using that
		$$
		U_0 \in \mathcal N \quad \Leftrightarrow \quad S(t)(U_0) = U_0, \quad t > 0,
		$$
		then $E_U(t)$  is a strict Lyapunov functional for the dynamical system $(\mathcal H,S^{\alpha}(t))$.
	\end{proof}

	\section{Entropy Estimates and Fractal Dimension}
	In this Section we reach an estimate for the Kolmogorov's $\epsilon$-entropy of the global attractor $\mathcal A_{\alpha}$ corresponding to the dynamical system $(\mathcal{H},S^{\alpha}(t))$. In order, let us formulate the definition of the Kolmogorov's $\varepsilon$-entropy of a compact set $\mathcal{M}$ in a Hilbert
	space $H$ (see \cite{chueshov-lasiecka-2005, Kolmogorov,Tikhomirov}).
	
	\begin{definition}\rm
		The  {\it Kolmogorov  $\varepsilon$- entropy} ${H}_{\varepsilon}(\mathcal{M})$ of a compact set $\mathcal{M}\subset H$   is given by
		\begin{align*}
			{H}_{\varepsilon}(\mathcal{M})=\ln N(\mathcal{M}, \varepsilon), \quad \varepsilon>0,
		\end{align*}	
		where $N(\mathcal{M}, \varepsilon)$ is the minimal number of closed sets of the diameter not greater than $2 \varepsilon $ which cover the compact $\mathcal{M}.$  The {\it fractal dimension} $\operatorname{dim}_{f} \mathcal{M}$ of $\mathcal{M}$ is defined by the formula
		$$
		\mbox{dim}^{f}_{H}\mathcal{M}=\limsup _{\varepsilon \rightarrow 0} \frac{{H}_{\varepsilon}(\mathcal{M})}{\ln (1 / \varepsilon)}.
		$$
	\end{definition}
	
	The next result that can be found in \cite{chueshov-lasiecka-2005} establishes an estimate for Kolmogorov's $\varepsilon$-entropy $H_{\varepsilon}(\mathcal{M})$ of a compact set $\mathcal{M}\subset H$.
	\begin{theorem}[{\cite[Theorem 4.2]{chueshov-lasiecka-2005}}]\label{theo-kolmog-entrop}
		Let $H$ be a separable Hilbert space and $\mathcal{M}$ be a bounded closed set in $H .$ Assume that there exists a mapping $V: \mathcal{M} \mapsto H$ such that:
		
		\begin{itemize}
			\item[$ 1. $] $\mathcal{M} \subseteq V \mathcal{M};$
			
			\item[$ 2. $]  $V$ is Lipschitz on $\mathcal{M},$ that is, there exists $L>0$ such that
			$$
			\big{\|}V z_{1}-V z_{2}\big{\|} \leq L\big{\|}z_{1}-z_{2}\big{\|}, \quad z_{1}, z_{2} \in \mathcal{M};
			$$
			
			\item[$ 3. $] There exist pseudometrics $\varrho_{1}$ and $\varrho_{2}$ on $H$ such that
			$$
			\big{\|}V z_{1}-V z_{2}\big{\|} \leq g\big{(}\big{\|}z_{1}-z_{2}\big{\|}\big{)}+h\big{(}\big{[}\varrho_{1}\big{(}z_{1}, z_{2}\big{)}^{2}+\varrho_{2}\big{(}V z_{1}, V z_{2}\big{)}^{2}\big{]}^{1 / 2}\big{)}
			$$
			for all $z_{1}, z_{2} \in \mathcal{M},$ where $g, h: \mathbb{R}^{+} \to \mathbb{R}^{+}$ are continuous non-decreasing functions such that
			$$
			g(0)=0, \ g(s)<s, \, s>0, \ s-g(s) \text { is nondecreasing,}
			$$
			and the function $h(s)$ is strictly increasing in the interval $\big{[}0, s_{0}\big{]}$ for some $s_{0}>0$ with  $h(0)=0.$

			\item[$ 4. $]For any $q>0$ and for any closed bounded set $B \subset \mathcal{M}$ the maximal number $m(B, q)$ of elements $x_{j}^{B} \in B$ such
			$$
			\varrho_{1}\big{(}x_{j}^{B}, x_{i}^{B}\big{)}^{2}+\varrho_{2}\big{(}V x_{j}^{B}, V x_{i}^{B}\big{)}^{2}>q^{2},  \ \ i \neq j,  \, i, j=1, \ldots, m(B, q),
			$$
			is finite.
		\end{itemize}
		
		Then $\mathcal{M}$ is a compact set and there exists $0<\varepsilon_{0}<1$ such that for all $\varepsilon \leq \varepsilon_{0}<1,$ Kolmogorov's $\varepsilon$-entropy ${H}_{\varepsilon}(\mathcal{M})$ admits the following estimate
		$$
		{H}_{\varepsilon}(\mathcal{M}) \leq \int_{\varepsilon}^{\varepsilon_{0}} \frac{\ln m\big{(}g_{\delta}^{-1}(s), q(s)\big{)}}{s-g_{\delta}(s)} d s+H_{g_{0}\big{(}\varepsilon_{0}\big{)}}(\mathcal{M}),
		$$
		where $g_{\delta}(s)=\frac{1-\delta}{2} g(2 s)+\delta s$ with arbitrary $\delta \in(0,1),$ the function $q(s)$ is defined by the formula
		$$
		q(s)=\frac{1}{2} h^{-1}\{\delta[2 s-g(2 s)]\}, 0<s<\varepsilon_{0},
		$$
		and
		$$
		m(r, q)=\sup \{m(B, q): B \subseteq \mathcal{M}, \operatorname{diam} B \leq 2 r\}.
		$$
	\end{theorem}
	
	\begin{theorem}[{\bf Kolmogorov $\varepsilon$-Entropy}]
		\label{theo-main2}
		Under Assumption \ref{Assumption}, there exists $0<\varepsilon_0<1$ such that,
		for all $\varepsilon\leq\varepsilon_0$, the Kolmogorov $\varepsilon$-entropy
		$H_\varepsilon(\mathcal{A}_\alpha)$ of the global attractor
		$\mathcal{A}_\alpha$ satisfies, for arbitrary $\delta\in(0,1)$,
		\begin{equation}
			\label{eps-entrop1}
			H_\varepsilon(\mathcal{A}_\alpha)
			\leq
			\frac{2}{1-\delta}
			\int_\varepsilon^{\varepsilon_0}
			\frac{\ln m\bigl(g_\delta^{-1}(s),q(s)\bigr)}{s}\,ds
			+
			H_{g_\delta(\varepsilon_0)}(\mathcal{A}_\alpha),
		\end{equation}
		where
		$g_\delta(s)=\frac{1+\delta}{2}s$ and
		$q(s)=\frac{1}{2}(\delta s)^{\frac{2(\alpha+1)}{\alpha+2}}$,
		$0<s<\varepsilon_0$, and
		\begin{equation*}
			m(r,a)
			=
			\sup\left\{
			m(B,a):\ B\subseteq\mathcal{A}_\alpha,\ 
			\operatorname{diam} B\leq 2r
			\right\},
		\end{equation*}
		with $m(B,a)$ denoting the maximal number of elements
		$U_j^B\in B$ such that, for any $a>0$,
		\begin{align*}
			\varrho\bigl(S^{\alpha}(T)U_j^B,S^{\alpha}(T)U_i^B\bigr)>a,
			\qquad i\neq j,\quad i,j=1,\ldots,m(B,a),
		\end{align*}
		for $T>0$ sufficiently large, where $\varrho$ is a pseudometric on
		$\mathcal{H}$.
	\end{theorem}
	
	\begin{proof}
		Let us consider two solution trajectories
		\begin{align*}
			S^\alpha(t)U_0^1
			=
			\bigl(u^1(t),u_t^1(t)\bigr),
			\qquad
			S^\alpha(t)U_0^2
			=
			\bigl(u^2(t),u_t^2(t)\bigr),
		\end{align*}
		corresponding to the initial data
		\begin{align*}
			U_0^1
			=
			\bigl(u_0^1,u_1^1\bigr),
			\qquad
			U_0^2
			=
			\bigl(u_0^2,u_1^2\bigr)
			\in\mathcal{A}_\alpha,
		\end{align*}
		and let $z:=u^1-u^2$.
		
		Since $\mathcal{A}_\alpha$ is compact (and hence bounded) and invariant
		under $S^\alpha(t)$, that is,
		$
		S^\alpha(t)\mathcal{A}_\alpha=\mathcal{A}_\alpha,$
		we have
		$S^\alpha(t)U_0^1,S^\alpha(t)U_0^2\in\mathcal{A}_\alpha$
		for all $t\geq0$.
		
		From Proposition~\ref{Prop-stab-est-0}, with $B=\mathcal{A}_\alpha$
		and $C_B=C_{\mathcal{A}_\alpha}$, there exists
		$T=T(\mathcal{A}_\alpha)>0$ such that
		\begin{align}
			\label{kolmog-1}
			\begin{aligned}
				\|S^\alpha(T)U_0^1-S^\alpha(T)U_0^2\|_{\mathcal{H}}
				&\leq
				\frac12\|U_0^1-U_0^2\|_{\mathcal{H}}
				\\
				&\quad+
				\left[
				C_{\mathcal{A}_\alpha}
				\sup_{0\leq s\leq T}
				\sqrt{
					\|z(s)\|^2
					+
					\|z_t(s)\|_{H^{-1}}^2
				}
				\right]^{\frac{\alpha+2}{2(\alpha+1)}},
			\end{aligned}
		\end{align}
		for all $U_0^1,U_0^2\in\mathcal{A}_\alpha$.
		
		Moreover, by Theorem~\ref{theo-global-solution}, there exists a constant $C>0$
		such that
		\begin{align}\label{kolmog-2}
			\|S^\alpha(T)U_0^1-S^\alpha(T)U_0^2\|_{\mathcal H}
			\leq
			C\|U_0^1-U_0^2\|_{\mathcal H},
			\qquad
			\forall\,U_0^1,U_0^2\in\mathcal A_\alpha.
		\end{align}
		
		Therefore, it follows from
		\eqref{kolmog-1}--\eqref{kolmog-2} that assumptions~1--4 of
		Theorem~\ref{theo-kolmog-entrop} are fulfilled with
		\begin{eqnarray*}
			\begin{array}{l}
				\mathcal{M}:=\mathcal{A}_\alpha,
				\qquad
				V:=S^\alpha(T),
				\qquad
				g(s)=\frac12s,
				\qquad
				h(s):=s^{\frac{\alpha+2}{2(\alpha+1)}},
				\qquad
				\varrho_1\equiv0,
				\\[2mm]
				\displaystyle
				\varrho_2\bigl(S^\alpha(T)U_0^1,S^\alpha(T)U_0^2\bigr)
				:=
				C_{\mathcal{A}_\alpha}
				\sup_{0\leq s\leq T}
				\sqrt{
					\|z(s)\|^2+\|z_t(s)\|_{H^{-1}}^2
				}.
			\end{array}
		\end{eqnarray*}
		It is worth noting that, due to the compact embedding
		$\mathcal{H}\hookrightarrow\hookrightarrow\mathcal{H}_{-1}$,
		the function $\varrho:=\varrho_2$ is a compact seminorm on $\mathcal{H}$.
		Indeed, if $X_n\rightharpoonup0$ weakly in $\mathcal{H}$, then
		$X_n\to0$ strongly in $\mathcal{H}_{-1}$, and the uniform continuity of
		the trajectories in $\mathcal{H}_{-1}$ yields
		$\varrho(X_n)\to0$. Thus, item~4 of
		Theorem~\ref{theo-kolmog-entrop} follows by a standard argument; see, e.g.,
		\cite[p.~55]{chueshov-lasiecka-2005}.
		Hence, the estimate for the Kolmogorov $\varepsilon$-entropy
		$H_\varepsilon(\mathcal{A}_\alpha)$ given in
		\eqref{eps-entrop1} follows from the conclusion of
		Theorem~\ref{theo-kolmog-entrop}.
	\end{proof}

	\begin{remark}\rm
		For $\alpha=0$, estimate~\eqref{asymptotic08} reduces to the
		quasi-stability estimate in the sense of \cite[Definition 7.9.2]{Chueshov}.
		Hence, all the hypotheses of \cite[Theorem 7.9.6]{Chueshov} are satisfied,
		and consequently the global attractor $\mathcal{A}_0$ has finite fractal
		dimension.
	\end{remark}
	
	\section{Smoothness of Global Attractors}
	\begin{remark}\rm
		As a direct consequence of Proposition~\ref{Prop-absorbing-set}, the global attractor
		$\mathcal A_\alpha$ is contained in the bounded absorbing set $\mathcal B$, namely,
		$
		\mathcal A_\alpha\subset\mathcal B.
		$
		Since $\mathcal A_\alpha$ is fully invariant under the semigroup
		$\{S^\alpha(t)\}_{t\ge0}$, every complete trajectory
		$\{U(t):t\in\mathbb R\}\subset\mathcal A_\alpha$
		remains in $\mathcal A_\alpha$ for all $t\in\mathbb R$. Consequently,
		$$
		U(t)\in\mathcal A_\alpha\subset\mathcal B,
		\qquad \forall\,t\in\mathbb R,
		$$
		which implies
		\begin{align}\label{unif-bound-H}
			\sup_{t\in\mathbb R}||U(t)||_{\mathcal H}\le R.
		\end{align}
		Hence, every complete trajectory lying on the global attractor is uniformly bounded in the phase space $\mathcal H$.
	\end{remark}
	
	Our next result shows that these trajectories are, in fact, uniformly bounded in the more regular phase space $\mathcal H_1$.
	\begin{theorem}\label{theo-regularity}{\bf[Regularity]}
		Under Assumption
		\ref{Assumption}, every complete trajectory
		$\{U(t):t\in\mathbb{R}\}$ contained in the global attractor satisfies
		\begin{align}\label{regularity-H1}
			\|U(t)\|_{\mathcal H_1}^2\le R_1^2,
		\end{align}
		for all $t\in\mathbb{R}$, where the constant $R_1$ is independent of $\alpha\ge0$.
	\end{theorem}
	\begin{proof}
		To establish the regularity of the family of attractors in $\mathcal{H}_1$, we follow the general strategy introduced by Chueshov and Lasiecka in
		\cite[Section~4.2]{chueshov-lasiecka-2007}; see also \cite{zelik},
		while adapting the arguments to the present wave equation with nonlocal Kelvin--Voigt damping.
		The proof is divided into three steps.
		First, we establish the $H^2$-regularity of the stationary solutions associated with
		\eqref{P}.
		Next, we use this property to derive the higher regularity of complete trajectories as
		$t\to-\infty$.
		Finally, given a complete trajectory satisfying
		$U(T_0)\in\mathcal H_1$ for some $T_0<0$, we show that this regularity propagates forward in time, yielding
		$U(t)\in\mathcal H_1$ for every $t\in\mathbb R$.
		
		\subsection*{Step 1: Regularity of Stationary Solutions}
		A stationary solution is a time-independent solution of problem \eqref{P}.
		Thus, it satisfies the associated elliptic problem
		\begin{equation}\label{stationary-problem}
			\left\{
			\begin{aligned}
				-\Delta u+f(u)&=h,
				&&\text{in }\Omega,\\
				u&=0,
				&&\text{on }\Gamma.
			\end{aligned}
			\right.
		\end{equation}
		Testing the stationary problem \eqref{stationary-problem} with
		$-\Delta u$ and integrating over $\Omega$, we obtain
		\begin{equation}\label{est-2}
			\|\Delta u\|^2
			-\int_{\Omega}f(u)\Delta u\,dx
			=
			-\int_{\Omega}h\,\Delta u\,dx.
		\end{equation}
		Using Assumption~\eqref{hyp_f'} and integrating by parts, we obtain
		$$
		-\int_{\Omega}f(u)\Delta u\,dx
		=
		\int_{\Omega}f'(u)|\nabla u|^2\,dx
		\ge
		-K_f\|\nabla u\|^2,
		$$
		where $K_f$ is the constant introduced in Remark~\ref{lowerbounded}; see \eqref{lower1}. Moreover, by H\"older's and Young's inequalities,
		$$
		\left|
		\int_{\Omega}h\,\Delta u\,dx
		\right|
		\le
		\frac12\|h\|^2
		+\frac12\|\Delta u\|^2.
		$$
		Combining the above estimates with \eqref{est-2}, we arrive at
		$$
		\|\Delta u\|^2
		\le
		\|h\|^2
		+
		2K_f\|\nabla u\|^2.
		$$
		Since the set of stationary solutions is contained in the global attractor $\mathcal A_\alpha$, applying \eqref{unif-bound-H} to the corresponding constant trajectories gives
		\begin{equation}\label{est-8}
			\|\Delta u\|^2
			\le
			\|h\|^2
			+
			2K_fC_R
			=: \varrho.
		\end{equation}
		Finally, the elliptic estimate yields
		$$
		\|u\|_{H^2(\Omega)}
		\le
		C\|\Delta u\|
		\le
		C\sqrt{\varrho},
		$$
		which proves that the set of stationary solutions is bounded in $H^2(\Omega)$.
		
		\subsection*{Step 2: Smoothness on negative time scale}
		For each $\alpha\ge0$, let
		$U(t)=(u(t),u_t(t))$, $t\in\mathbb{R}$, be a complete trajectory on the global attractor
		$\mathcal{A}_{\alpha}=\mathcal{M}^{u}_{\alpha}(\mathcal{N})$. We write
		$u=w+z$, where $W=(w,w_t)$ is chosen as the solution of the following problem:
		\begin{eqnarray}\left\{\begin{array}{l}\label{prob-H2-sol-1}
				w_{tt} -\Delta w - \|\nabla w_t\|^{\alpha}\Delta w_t + f(w) = h,\quad\mbox{in}\quad\Omega \times (s,T)  \\
				w=0,\quad \mbox{on}\quad\Gamma,\\
				W(s)==(u^*,0),
			\end{array}\right.\end{eqnarray}
		where {\bf $u^*\in H^2$ is a stationary solution of \eqref{P}}, and $Z=(z,z_t)$ satisfying the following problem
		\begin{eqnarray}\left\{\begin{array}{l}\label{prob-H2-sol-2}
				z_{tt} -\Delta z -\left[\,\|\nabla u_t\|^{\alpha}\Delta u_t-\|\nabla w_t\|^{\alpha}\Delta w_t\right]=-[\,f(u)-f(w)\,], \quad\mbox{in}\quad\Omega \times (s,T),\\
				z=0,\quad\mbox{on}\quad\Gamma,\\
				(z(s),z_t(s)) = (u(s)-u^*,u_t(s)).
			\end{array}\right.\end{eqnarray}
		\paragraph{\it {Estimate 1}.}
		The estimates below are formal, but they can be rigorously justified by means of a standard Galerkin approximation argument. Taking the inner product in $L^2(\Omega)$ of the equation in \eqref{prob-H2-sol-1} with $-\Delta w_t$, we obtain
		\begin{eqnarray}\label{reg-H2-1}
			\begin{aligned}
				&\frac{d}{dt}\left[
				\frac12\|W\|_{\mathcal H_1}^2
				+\frac12\int_\Omega f'(w)|\nabla w|^2\,dx
				+\int_\Omega h\,\Delta w\,dx
				\right]
				+\|\nabla w_t\|^\alpha\|\Delta w_t\|^2
				\\
				&\qquad
				=
				\frac12\int_\Omega
				f''(w)w_t|\nabla w|^2\,dx.
			\end{aligned}
		\end{eqnarray}
		Next, taking the inner product in $L^2(\Omega)$ of the equation in \eqref{prob-H2-sol-1} with $-\beta\Delta w$, where $\beta>0$, yields
		\begin{eqnarray}\label{reg-H2-2}
			\begin{aligned}
				&\beta\frac{d}{dt}
				(\nabla w_t,\nabla w)
				+\beta\|\Delta w\|^2
				+\beta\int_\Omega
				f'(w)|\nabla w|^2\,dx
				+\beta\int_\Omega
				h\,\Delta w\,dx
				\\
				&\qquad
				=
				\beta\|\nabla w_t\|^2
				-\beta\|\nabla w_t\|^\alpha
				\int_{\Omega}\nabla w\nabla w_t\,dx.
			\end{aligned}
		\end{eqnarray}
		Introducing the Lyapunov functional
		$$
		E_W(t):=
		\frac12\|W\|_{\mathcal H_1}^2
		+\frac12\int_\Omega
		f'(w)|\nabla w|^2\,dx
		+\int_\Omega
		h\,\Delta w\,dx
		+\beta\int_\Omega
		\nabla w_t\nabla w\,dx,
		$$
		and combining \eqref{reg-H2-1} and \eqref{reg-H2-2}, we obtain
		\begin{eqnarray}\label{reg-H2-12}
			\begin{aligned}
				&\frac{d}{dt}E_W(t)
				+\beta E_W(t)
				+\frac{\beta}{2}\|\Delta w\|^2
				+\|\nabla w_t\|^\alpha\|\Delta w_t\|^2
				\\
				&\qquad
				=
				\frac{3\beta}{2}\|\nabla w_t\|^2
				+\frac12\int_\Omega
				f''(w)w_t|\nabla w|^2\,dx
				\\
				&\qquad\quad
				-\beta\|\nabla w_t\|^\alpha
				\int_\Omega
				\nabla w\nabla w_t\,dx
				+\beta^2
				\int_\Omega
				\nabla w_t\nabla w\,dx.
			\end{aligned}
		\end{eqnarray}
		Note that, from dissipative condition \ref{hyp-inf-f}, Holder inequality, and embedding $H^2(\Omega)\hookrightarrow H^1_0(\Omega)$, we can estimate the functional $E_W(t)$ from below as follows
		\begin{eqnarray}\label{reg-H2-3}
			\begin{aligned}E_W(t)\ge &\;\frac{1}{4}\|\Delta w\|^2+\frac{1}{2}\|\nabla w_t\|^2-\|h\|^2-\frac{K_f+\beta^2}{2}\|\nabla w\|^2,
		\end{aligned}\end{eqnarray}
		where $K_f$ is given by \eqref{lower1} in Remark~\ref{lowerbounded}. So, we define the perturbed functional $\widetilde{E}_W(t)$ by
		\begin{align*}\widetilde{E}_W(t):=E_W(t)+\frac{K_f+\beta^2}{2}\|\nabla w\|^2+\|h\|^2,\end{align*}
		it follows from \eqref{reg-H2-3} that
		\begin{eqnarray}\label{reg-H2-33}\widetilde{E}_W(t)\ge \frac{1}{4}||W||^2_{\mathcal{H}_1}.
		\end{eqnarray}
		Now, returning to \eqref{reg-H2-12}, we obtain
		\begin{eqnarray}\label{reg-H2-123}
			\begin{aligned}
				&\frac{d}{dt}\widetilde{E}_W(t)+\beta\widetilde{E}_W(t)+\frac{\beta}{2}\|\Delta w\|^2+\|\nabla w_t\|^{\alpha}\|\Delta w_t\|^2\\
				&= \beta\|h\|^2+\frac{3\beta}{2}\|\nabla w_t\|^2+\beta\int_{\Omega}\nabla w_t\nabla wdx+(K_f+\beta^2)\int_{\Omega}\nabla w_t\nabla wdx\\
				&\quad+\frac{\beta(K_f+\beta^2)}{2}\|\nabla w\|^2-\beta\|\nabla w_t\|^{\alpha}\int_{\Omega}\nabla w\nabla w_tdx+\frac{1}{2}\int_{\Omega}f''(w)w_t|\nabla w|^2dx.
			\end{aligned}
		\end{eqnarray}
		Next, we estimate the terms on the right-hand side of \eqref{reg-H2-123}. First, by Poincar\'e's inequality,
		$$
		\lambda_1\|\nabla w_t\|^2\le \|\Delta w_t\|^2,
		$$
		which implies that
		$$
		\lambda_1\|\nabla w_t\|^{\alpha+2}
		=\lambda_1\|\nabla w_t\|^\alpha\|\nabla w_t\|^2
		\le
		\|\nabla w_t\|^\alpha\|\Delta w_t\|^2.
		$$
		On the other hand, using H\"older's and Young's inequalities together with the uniform estimate \eqref{unif-bound-H}, we obtain
		\begin{eqnarray*}
			&& \beta\|h\|^2+\frac{3\beta}{2}\|\nabla w_t\|^2+\beta\int_{\Omega}\nabla w_t\nabla w\,dx+(K_f+\beta^2)\int_{\Omega}\nabla w_t\nabla w\,dx\\
			&&\quad+\frac{\beta(K_f+\beta^2)}{2}\|\nabla w\|^2-\beta\|\nabla w_t\|^{\alpha}\int_{\Omega}\nabla w\nabla w_t\,dx\\
			&&\le \lambda_1\|\nabla w_t\|^{\alpha+2}+C\left[1+\|\nabla w\|^{\frac{\alpha+2}{\alpha+1}}+\|\nabla w\|^{2}+\|\nabla w\|^{\alpha+2}\right]\le \lambda_1\|\nabla w_t\|^{\alpha+2}+C_{R}.
		\end{eqnarray*}
		Finally, by Assumption \eqref{hyp_f''}, H\"older's inequality with
		$\frac{1}{2}+\frac{1}{6}+\frac{1}{3}=1$, the embeddings
		$H^2(\Omega)\hookrightarrow H^1_0(\Omega)\hookrightarrow L^6(\Omega)$,
		the uniform bound \eqref{unif-bound-H}, and \eqref{reg-H2-33}, we have
		\begin{eqnarray*}
			\frac{1}{2}\int_{\Omega}f''(w)w_t|\nabla w|^2\,dx
			&\le&
			C\left(1+\|w\|_6^3\right)\|w_t\|_6\|\nabla w\|_6^2\\
			&\le&
			C_R\|\nabla w_t\|\|\Delta w\|^2\le
			\frac{\beta}{2}\widetilde{E}_W(t)
			+d(t,w)\widetilde{E}_W(t),
		\end{eqnarray*}
		where
		$
		d(t,w):=C_R\|\nabla w_t(t)\|^{\alpha+2}.
		$
		
		Thus, using the last three inequalities, we conclude from \eqref{reg-H2-123} that
		\begin{eqnarray}\label{reg-H2-4}
			\frac{d}{dt}\widetilde{E}_W(t)+\frac{\beta}{2}\widetilde{E}_W(t)\le d(t,w)\widetilde{E}_W(t)+C_{R},
		\end{eqnarray}
		Multiplying \eqref{reg-H2-4} by the integrating factor $e^{\frac{\beta}{2}t}$ and integrating from $s$ to $t$, we obtain
		\begin{eqnarray}\label{reg-H2-6}
			\begin{aligned}
				\widetilde{E}_W(t)e^{\frac{\beta}{2}t}\le e^{\frac{\beta}{2}s}\widetilde{E}_w(s)+\int_s^td(\tau,w)e^{\frac{\beta}{2}\tau}\widetilde{E}_w(\tau)d\tau+ \frac{2C_{R}}{\alpha}e^{\frac{\beta}{2}t}.
			\end{aligned}
		\end{eqnarray}
		Using that $d(\cdot,w)\in L^1(s,t)$ and \eqref{reg-H2-33}, it follows from \eqref{reg-H2-6} that
		\begin{eqnarray}\label{reg-H2-7}
			\begin{aligned}
				||W||^2_{\mathcal{H}_1}e^{\frac{\beta}{2}t}\le 4e^{\|d\|_{L^1}}\left[\, e^{\frac{\beta}{2}s}\widetilde{E}_W(s)+\frac{ 2C_{R}}{\beta}e^{\frac{\beta}{2}t}\right].
			\end{aligned}
		\end{eqnarray}
		Note that, from \eqref{est-8} [regularity of steady states], we have
		\begin{align}\label{reg-H2-8}\widetilde{E}_W(s)\le C\|\Delta u^*\|^2+\frac{1}{2}\|h\|^2\le C\varrho+\frac{1}{2}\|h\|^2=:\varrho_0.\end{align}
		Thus, taking the limit in \eqref{reg-H2-7} with $s\to -\infty$, we obtain
		
		\begin{eqnarray}\label{reg-H2-9}
			\begin{aligned}
				||W||^2_{\mathcal{H}_1}\le \frac{8C_{R} e^{\|d\|_{L^1(s,t)}}}{\beta}.
			\end{aligned}
		\end{eqnarray}
		As for \eqref{reg-H2-8}, $\widetilde{E}_W(s)\le \varrho_0$ independent on $s$, we conclude from \eqref{reg-H2-9} that the full trajectory $\{W(t)=(w(t),w_t(t)):t\in \mathbb{R}\}$ has $\mathcal{H}_1$- regularity.
		
		\bigskip
		\paragraph{\it{Estimate 2}.}
		Our goal now is to show that the original trajectory  $U(t)$ coincides with $W(t)$ for $t\in (-\infty, T_0]$ for some $T_0<0$. In fact, multiplying \eqref{prob-H2-sol-2} by $z_t$ and integrating over $\Omega\times[s,t]$, we have
		\begin{eqnarray}\label{est-z-1}
			\begin{aligned}
				&E_Z(t)-\int_s^t\int_{\Omega}\left[\,\|\nabla u_t\|^{\alpha}\Delta u_t-\|\nabla w_t\|^{\alpha}\Delta w_t\,\right]z_tdxd\tau=E_z(s)\\
				&\quad+\underbrace{K_f\int_s^t\int_{\Omega}z_tzdxd\tau}_{\mathrm{J}_1}+\underbrace{\frac{1}{2}\int_s^t\int_{\Omega}\int_0^1f''(\theta u+(1-\theta)w)[u_t+(1-\theta)w_t]d\theta |z|^2dxd\tau}_{\mathrm{J}_2}.
			\end{aligned}
		\end{eqnarray}
		where
		$$E_Z(t)=\frac{1}{2}\|z_t\|^2+\frac{1}{2}\|\nabla z\|^2+\frac{1}{2}\int_{\Omega}\int_0^1f'(\theta u+(1-\theta)w)d\theta|z|^2dx+\frac{K_f}{2}\|z\|^2.$$
		
		Note that
		\begin{align*}
			&-\int_s^t\int_{\Omega}
			\left(\|\nabla u_t\|^{\alpha}\Delta u_t
			-\|\nabla w_t\|^{\alpha}\Delta w_t\right)
			z_t\,dx\,d\tau\\
			&=
			-\frac12\int_s^t
			\left(\|\nabla u_t\|^{\alpha}
			+\|\nabla w_t\|^{\alpha}\right)
			\int_{\Omega}\Delta z_t\,z_t\,dx\,d\tau\\
			&\qquad
			-\frac12\int_s^t
			\left(\|\nabla u_t\|^{\alpha}
			-\|\nabla w_t\|^{\alpha}\right)
			\int_{\Omega}\Delta(u_t+w_t)z_t\,dx\,d\tau
			\\
			&=
			\frac12\int_s^t
			\left(\|\nabla u_t\|^{\alpha}
			+\|\nabla w_t\|^{\alpha}\right)
			\|\nabla z_t\|^2\,d\tau\\
			&\qquad
			+\frac12\int_s^t
			\left(\|\nabla u_t\|^{\alpha}
			-\|\nabla w_t\|^{\alpha}\right)
			\left(\|\nabla u_t\|^2
			-\|\nabla w_t\|^2\right)\,d\tau
			\ge 0.
		\end{align*}
		Returning to \eqref{est-z-1}, we obtain
		\begin{eqnarray}\label{est-z-11}
			\begin{aligned}
				E_Z(t)\le E_Z(s)+\mathrm{J}_1+\mathrm{J}_2.
			\end{aligned}
		\end{eqnarray}
		By Poincar\'e's inequality, we have
		\begin{align*}
			\left|\mathrm{J}_1\right|
			\le
			\frac{K_f}{\lambda_1^{1/2}}
			\int_s^tE_Z(\tau)\,d\tau.
		\end{align*}
		Moreover, by H\"older's inequality with
		$\frac{2}{3}+\frac{1}{6}+\frac{1}{6}=1$,
		the embedding
		$H^1_0(\Omega)\hookrightarrow L^6(\Omega)$,
		and the uniform estimate \eqref{unif-bound-H}, we have
		\begin{align*}
			\left|\mathrm{J}_2\right|
			\le
			C_R\int_s^tE_Z(\tau)\,d\tau.
		\end{align*}
		Substituting the above estimates into \eqref{est-z-11}, we obtain
		\begin{align}\label{est-z-4}
			E_Z(t)\le E_Z(s)+C'_R\int_s^tE_Z(\tau)\,d\tau.
		\end{align}
		Applying Gronwall's lemma to \eqref{est-z-4}, we obtain
		\begin{align}\label{est-z-44}
			E_Z(t)\le e^{C'_R(t-s)}E_Z(s).
		\end{align}
		Taking $t=s+r$, with $r>0$, it follows from \eqref{est-z-44} that
		\begin{align}\label{est-z-5}
			E_Z(s+r)\le e^{C'_Rr}E_Z(s).
		\end{align}
		Since
		\begin{align*}
			\frac12\|Z(t)\|_{\mathcal H}^2
			\le
			E_Z(t)
			\le
			C\|Z(t)\|_{\mathcal H}^2,
		\end{align*}
		we deduce from \eqref{est-z-5} that
		\begin{align*}
			\|Z(s+r)\|_{\mathcal H}^2
			\le
			Ce^{C'_Rr}\|Z(s)\|_{\mathcal H}^2.
		\end{align*}
		Since $Z(s)\to0$ in $\mathcal H$ as $s\to-\infty$, we conclude that
		\begin{align*}
			\|Z(s+r)\|_{\mathcal H}\longrightarrow0,
			\qquad s\to-\infty,
		\end{align*}
		for every fixed $r>0$.
		Hence, for every $\varepsilon>0$, there exists $T_\varepsilon<0$ such that
		\begin{align*}
			\|Z(t)\|_{\mathcal H}<\varepsilon,
			\qquad \forall\,t\le T_\varepsilon.
		\end{align*}
		
		\subsection*{Step 3: Forward propagation of the regularity}
		
		Let us now consider problem \eqref{P} with initial data
		$U(T_0)\in \mathcal{H}_1$ for some $T_0<0$, as obtained in Step 2. Namely,
		\begin{eqnarray*}\left\{\begin{array}{l}
				u_{tt} -\Delta u -\|\nabla u_t\|^{\alpha}\Delta u_t + f(u) = h,\quad\mbox{in}\quad\Omega \times (T_0,\infty),\\
				u=0,\quad \mbox{on}\quad\Gamma,\\
				(u(T_0),u_t(T_0))=(u^*_0,u^*_1)\in \mathcal{H}_1.
			\end{array}\right.\end{eqnarray*}
		Proceeding exactly as in Estimate 1 of Step 2, we obtain the following inequality
		(analogous to \eqref{reg-H2-7})
		\begin{align*}
			||U(t)||^2_{\mathcal{H}_1}
			\leq
			4e^{\|g\|_{L^1(T_0,t)}}e^{-\varpi(t-T_0)}
			\widetilde{E}_{U}(T_0)
			+
			4e^{\|g\|_{L^1(T_0,t)}}K_{R},
			\qquad \forall t\geq T_0,
		\end{align*}
		where $\varpi>0$,
		$g(t,u)=C\|\nabla u_t(t)\|^{\alpha+2}$, and
		\begin{align*}
			\widetilde{E}_{U}(T_0)
			\leq C\|U(T_0)\|_{\mathcal H_1}^2<\infty.
		\end{align*}
		Since $d(\cdot,u)\in L^1(T_0,\infty)$, the factor
		$e^{\|d\|_{L^1(T_0,t)}}$ remains uniformly bounded for all $t\geq T_0$.
		Therefore, the above estimate implies that
		\begin{align*}
			\sup_{t\geq T_0}||U(t)||^2_{\mathcal{H}_1}<\infty.
		\end{align*}
		Consequently, $U$ is uniformly bounded in $\mathcal{H}_1$ for all
		$t\in[T_0,\infty)$. Combining Steps 2 and 3, we conclude that
		\begin{align*}
			U\in L^{\infty}(\mathbb R;\mathcal{H}_1),
		\end{align*}
		and there exists $R_1>0$ such that
		\begin{align*}
			||U(t)||^2_{\mathcal{H}_1}\leq R_1,
			\qquad \forall t\in\mathbb R.
		\end{align*}
		This proves \eqref{regularity-H1} and, consequently, establishes the smoothness of the global attractor $\mathcal{A}_{\alpha}$. The proof of Theorem~\ref{theo-regularity} is complete.
	\end{proof}

	\begin{remark}\label{rem-utt}
		By using equation~\eqref{P} and the uniform estimate in
		Theorem~\ref{theo-regularity}, we also obtain
		\[
		\sup_{\alpha\ge0}
		\|u_{tt}^{\alpha}\|_{L^\infty(-T,T;H^{-1}(\Omega))}
		\le C_T,
		\qquad \forall\,T>0,
		\]
		where $C_T$ is independent of $\alpha$.
	\end{remark}

	\section{Upper Semicontinuity}
	To conclude the paper, we establish the upper semicontinuity of the family of global attractors
	$\{\mathcal{A}_{\alpha}\}_{\alpha\ge0}$
	at $\alpha=0$ in the phase space
	$\mathcal{H}=H_0^1(\Omega)\times L^2(\Omega)$.
	This result shows that the global attractors associated with the wave equation endowed with nonlinear nonlocal Kelvin--Voigt damping (the case $\alpha>0$) converge, as $\alpha\to0$, to the global attractor corresponding to the wave equation with linear Kelvin--Voigt damping (the case $\alpha=0$).

	Since we are interested in the upper semicontinuity of the global attractor
	$\mathcal{A}_{\alpha}$ at $\alpha=0$, we restrict the parameter to
	$\Lambda=[0,1]$ without loss of generality.

	The upper semicontinuity of the family
	$\{\mathcal{A}_{\alpha}\}_{\alpha\in\Lambda}$ at $\alpha=0$ can now be
	stated as follows.
	
	\begin{theorem}{\bf[Upper semicontinuity]}\label{upper-semicontinuity}
		Let Assumption~\ref{Assumption} hold. Then the family of global attractors
		$\{\mathcal{A}_{\alpha}\}_{\alpha\in\Lambda}$ associated with problem~\eqref{P}
		in the phase space $\mathcal{H}$ is upper semicontinuous at $\alpha=0$, in the sense that
		\begin{align*}
			\lim_{\alpha\to0}
			\sup_{U\in\mathcal{A}_{\alpha}}
			\operatorname{dist}_{\mathcal{H}}(U,\mathcal{A}_{0})
			=0.
		\end{align*}
	\end{theorem}
	\begin{proof}
		Suppose, by contradiction, that the upper semicontinuity property does not
		hold. Then there exist $\delta>0$, a sequence $\alpha_n\to0$, and
		$U_n\in\mathcal A_{\alpha_n}$ such that
		\begin{align}\label{contradiction}
			\operatorname{dist}_{\mathcal H}(U_n,\mathcal A_0)\ge\delta,
			\qquad n\in \mathbb{N}.
		\end{align}
		For each $n$, let
		$\{U_n(t)=(u_n(t),u_{n,t}(t)):t\in\mathbb R\}$, 
		be a complete trajectory on $\mathcal A_{\alpha_n}$ such that
		$U_n(0)=U_n$. By Theorem~\ref{theo-regularity},
		\begin{align}\label{uniform-H1}
			\sup_{n\ge1}\sup_{t\in\mathbb R}
			\|U_n(t)\|_{\mathcal H_1}\le R_1,
		\end{align}
		where $R_1$ is independent of $n$. Moreover, by Remark~\ref{rem-utt},
		the corresponding time derivatives are uniformly bounded in
		$H^{-1}(\Omega)$. Therefore, by the Aubin--Lions compactness theorem, for every $T>0$,
		there exist a subsequence, still denoted by $\{U_n\}$, and a function
		$
		U(t)=(u(t),u_t(t))\in C_{\mathrm{bnd}}(\mathbb R;\mathcal H)
		$
		such that
		\begin{align}\label{compact-convergence}
			\max_{t\in[-T,T]}
			||U_n(t)-U(t)||_{\mathcal H}\longrightarrow0
			\qquad\text{as }n\to\infty.
		\end{align}
		Using the variational formulation~\eqref{variational-formula} of
		problem~\eqref{P}, together with $\alpha_n\to0$ and the uniform estimates
		above, we can pass to the limit and conclude that $U(t)$ is a complete
		trajectory of the limiting problem corresponding to $\alpha=0$. Furthermore, by~\eqref{uniform-H1},
		$$
		\|U(t)\|_{\mathcal H}\le R,
		\qquad t\in\mathbb R,
		$$
		for some $R>0$ independent of $t$. Hence, $U(t)$ is a bounded complete
		trajectory of the limiting system. Consequently,
		$
		U(t)\in\mathcal A_0,$ $t\in\mathbb R.
		$
		In particular,
		$
		U(0)\in\mathcal A_0.
		$
		Taking $t=0$ in~\eqref{compact-convergence}, we obtain
		\[
		U_n=U_n(0)\longrightarrow U(0)
		\qquad\text{in }\mathcal H.
		\]
		Since $U(0)\in\mathcal A_0$, this contradicts~\eqref{contradiction}.
		Therefore,
		\[
		\lim_{\alpha\to0}
		\operatorname{dist}_{\mathcal H}(\mathcal A_\alpha,\mathcal A_0)=0.
		\]
		This completes the proof.
	\end{proof}
	
	\section*{Declarations}
	
	\noindent\textbf{Conflict of Interest.}
	On behalf of all authors, the corresponding author declares that there is
	no conflict of interest.
	
	\medskip
	
	\noindent\textbf{Author Contributions.}
	All authors contributed to the conception and design of the study.
	
	\medskip
	
	\noindent\textbf{Data Availability.}
	Data sharing is not applicable to this article, as no datasets were
	generated or analyzed during the current study.

\end{document}